\documentclass[8pt]{article}
\usepackage{indentfirst,latexsym,bm}
\usepackage{amsfonts}
\usepackage{amssymb}
\usepackage{times}
\usepackage[leqno]{amsmath}
\usepackage{dsfont}
\usepackage[all]{xy}
\usepackage{color,xcolor}
\usepackage{amsthm}
\usepackage{hyperref}
\begin{document}
\title{Pre-modular  fusion categories of global dimensions $p^2$}
\author{Zhiqiang Yu}
\date{}
\maketitle

\newtheorem{theo}{Theorem}[section]
\newtheorem{prop}[theo]{Proposition}
\newtheorem{lemm}[theo]{Lemma}
\newtheorem{coro}[theo]{Corollary}
\theoremstyle{definition}
\newtheorem{defi}[theo]{Definition}
\newtheorem{exam}[theo]{Example}
\newtheorem{rema}[theo]{Remark}
\newtheorem{ques}[theo]{Question}

\newcommand{\A}{\mathcal{A}}
\newcommand{\ad}{\text{ad}}
\newcommand{\B}{\mathcal{B}}
\newcommand{\C}{\mathcal{C}}
\newcommand{\D}{\mathcal{D}}
\newcommand{\E}{\mathcal{E}}
\newcommand{\I}{\mathcal{I}}
\newcommand{\FC}{\mathbb{C}}
\newcommand{\FQ}{\mathbb{Q}}
\newcommand{\FPdim}{\text{FPdim}}
\newcommand{\Gal}{\text{Gal}}
\newcommand{\Gr}{\text{Gr}}
\newcommand{\ord}{\text{ord}}
\newcommand{\pt}{\text{pt}}
\newcommand{\Q}{\mathcal{O}}
\newcommand{\rank}{\text{rank}}
\newcommand{\hxs}{\mathfrak{s}}
\newcommand{\ssl}{\mathfrak{sl}}
\newcommand{\sso}{\mathfrak{so}}
\newcommand{\SL}{\text{SL}}
\newcommand{\hxt}{\mathfrak{t}}
\newcommand{\W}{\mathcal{W}}
\newcommand{\vvec}{\text{Vec}}
\newcommand{\Y}{\mathcal{Z}}
\newcommand{\YL}{\text{YL}}
\newcommand{\Z}{\mathbb{Z}}
\theoremstyle{plain}

\abstract
Let $p\geq5$ be a prime, we show that a non-pointed modular fusion category $\C$ is Grothendieck equivalent to $\C(\ssl_2,2(p-1))_A^0$ if and only if $\dim(\C)=p\cdot u$, where $u$ is a certain totally positive algebraic unit and $A$ is the regular algebra of the Tannakian subcategory $\text{Rep}(\Z_2)\subseteq\C(\ssl_2,2(p-1))$. As a direct corollary, we classify non-simple modular fusion categories of   global dimensions $p^2$.



\section{Introduction}\label{introduction}


A fusion category $\C$ over the complex field $\FC$ is a semisimple finite tensor category, fusion categories are widely studied with  restrictions on their Frobenius-Perron dimensions \cite{EGNO,ENO1}, ranks (i.e., the number of the isomorphism classes of simple objects) \cite{BNRW2,O2} etc.  The global dimension of a fusion category is an important concept that deserved more attentions.
Unlike the classification of fusion categories by Frobenius-Perron dimensions,  however, we even don't know the structures of   fusion categories  with a prime global dimension.

Recently,  some progresses are made in topics related to the classifications of spherical fusion categories by their global dimensions. In \cite{O3}, V. Ostrik  gave lower bounds of the dimensions (more generally, the formal codegrees) of fusion categories, as a direct result, he classified spherical fusion categories of integer global dimensions less than $6$. Braided spherical (or, pre-modular) fusion categories of global dimension less than or equal to 10 were completely classified by the author in \cite{Yu}, and  spherical fusion categories of dimension $6$ were also shown to be weakly integral. It was   conjectured in \cite{Yu}   that  pre-modular fusion categories of prime dimension $p$ are always pointed if $p\neq5$, this is answered affirmatively in \cite{Schopieray} lately. Moreover, by using a classical Siegel's trace theorem about totally positive algebraic integers \cite{Siegel}, spherical fusion categories of prime dimension $p$  are proved to be pointed if $(p-1)/2$  is also an odd prime.

As a special class of fusion categories,   modular fusion categories connect deeply with conformal field theory \cite{BK}, vertex operator algebras \cite{DongLNg}, quantum groups at root of unity \cite{BK,EGNO}. The $S$-matrix and $T$-matrix of modular fusion categories (see section \ref{preliminaries}), which  reflect many  important properties of modular fusion categories, also enjoy interesting algebraic and arithmetic properties \cite{BNRW1,BNRW2,DongLNg}. Therefore, modular fusion categories are  inseparable with algebraic number theory and representations of the modular group $\SL(2,\Z)$ \cite{BNRW1,BNRW2,EGNO},  in particular,  one can  peer into their properties   by considering  the number of Galois orbits of the simple objects \cite{NgWaZh,PSYZ} and the representation type of $\SL(2,\Z)$ associated to a modular category \cite{NgRWW}, for example.

Let $p$ be a prime, $\C$ a  modular fusion category of global dimension $p^2$.
Then $\C$ is tensor equivalent to an Ising category $\C(\ssl_2,2)$ or $\C$ is pointed \cite{O3} if $p=2$. Modular fusion categories of global dimension $9$ are either pointed or braided tensor equivalent to a Galois conjugate of the non-pointed simple modular fusion category
$\C(\sso_5,\frac{3}{2})_\text{ad}$ \cite[Theorem 4.8]{Yu}. Let $\C(\ssl_2,3)_\text{ad}$ denote the  Yang-Lee (or, Fibonacci) fusion category, which is a rank $2$ transitive modular fusion category \cite{NgWaZh}, and its  Galois conjugate $\C(\ssl_2,3)_\text{ad}^\sigma$ has global dimension $\frac{5-\sqrt{5}}{2}$, where $\sigma\in\text{Gal}(\FQ(\zeta_5)/\FQ)$ such that $\sigma(\zeta_5)=\zeta_5^2$. So  modular fusion categories
\begin{align*}
\C(\ssl_2,3)_\text{ad}\boxtimes\C(\ssl_2,3)_\text{ad}^\sigma\boxtimes \C(\ssl_2,3)_\text{ad}\boxtimes\C(\ssl_2,3)_\text{ad}^\sigma,
~\C(\ssl_2,3)_\text{ad}\boxtimes\C(\ssl_2,3)_\text{ad}^\sigma\boxtimes\C(\Z_5,\eta)
 \end{align*}
both have  global dimension $25$, where $\C(\Z_5,\eta)$ is a pointed modular fusion category of dimension $5$. When $p=7$, modular fusion category
$\C(\ssl_2,5)_\text{ad}\boxtimes\C(\ssl_2,5)_\text{ad}^\tau\boxtimes\C(\ssl_2,5)_\text{ad}^{\tau^2}$
and its Galois conjugates have global dimension $49$, where $\tau\in\Gal(\FQ(\zeta_7))$ such that $\tau(\zeta_7)=\zeta_7^2$, $\C(\ssl_2,5)_\text{ad}$ is a  transitive modular fusion categories of rank $3$ \cite{O2,NgWaZh}. It was asked in  \cite[Question 4.9]{Yu}  whether the modular fusion categories (and their Galois conjugates) mentioned above are all non-pointed  modular fusion categories of global dimension  $p^2$.

In this paper, from both the views of algebraic and arithmetic properties of modular fusion categories, we continue to investigate  the structure of (pre-)modular fusion categories of global dimensions $p^2$ with $p\geq5$. Let $\C$ be such a modular fusion category, if $\C$ does not contains a non-trivial fusion subcategory of integer dimension, we show that $\C$ always contains a Galois conjugate of the transitive modular fusion category $\C(\ssl_2,p-2)_\ad$ (Theorem \ref{nintdimfsb}), then  we give a complete classification of non-simple modular fusion category of global dimension $p^2$ (Theorem \ref{nonsimplep^2}). Moreover, we find that there exists another non-pointed modular fusion category of global dimension $11^2$, which then gives a negative answer to \cite[Question 4.9]{Yu}.



\begin{table}[htb]
\begin{center}
\caption{Some Notations}
\begin{tabular}{|c|c|}
\hline   \textbf{Notation} & \textbf{Meaning}  \\
\hline   $\zeta_n $& the $n$-th primitive root of unity $e^\frac{2\pi i}{n} $  \\
\hline   $N(f)$ & the norm of $f$, i.e., $N(f)=\Pi_{\sigma\in\Gal(\FQ(f)/\FQ)}\sigma(f)$  \\
\hline  $\Q_X(\C)$ & the Galois orbit of simple object $X$ of a modular fusion category $\C$\\
\hline  $\C(\mathfrak{g},k)$ & the modular fusion category obtained from representation category\\

& $\text{Rep}(U_q(\mathfrak{g}))$ of quantum group $U_q(\mathfrak{g})$ at root of unity\\
\hline
\end{tabular}
\end{center}
\end{table}
The paper is organized as follows. In section \ref{preliminaries}, we recall some basic notions and notations of  (modular)  fusion categories, such as  global dimensions, formal codegrees and $d$-numbers,  modular data, and the congruence representations of the modular group $\SL(2,\Z)$. In section \ref{section3}, we consider modular fusion categories whose norm of global dimensions are powers of a prime $p$, in particular, if $p\geq5$, we show that a non-pointed modular fusion category $\C$ is Grothendieck equivalent to $\C(\ssl_2,2(p-1))_A^0$ if and only if $\dim(\C)=p\cdot u$ in Theorem \ref{rank(p+3)/2} and  Corollary \ref{ptimesunit}, where $u$ is a certain algebraic unit. In section \ref{section4}, we first show that a non-simple modular fusion category $\C$ of global dimension $p^2$ contains a Galois conjugate of a transitive modular fusion subcategory  if $\C$ does not contain a non-trivial fusion subcategory with an integer global dimension in Theorem \ref{nintdimfsb}, then we give a complete classification of non-simple modular fusion categories of dimension $p^2$ in Theorem \ref{nonsimplep^2}.
\section{Preliminaries}\label{preliminaries}
In this section, we recall we will recall some important definitions and properties about  fusion categories and modular fusion categories, we refer the readers
  to \cite{BK,BNRW2,DrGNO2,EGNO,ENO1,Mu1,O1}.



  \subsection{Fusion categories and dimensions}
Let $\C$ be a fusion category,   $\Q(\C)$  the set of isomorphism classes of simple objects of a fusion category $\C$. Then the Frobenius-Perron homomorphism $\FPdim(-):\Gr(\C)\to\FC$ is the unique ring homomorphism such that $\FPdim(X)\geq1$ is an algebraic integer, $\forall X\in\Q(\C)$, $\FPdim(X)$ is called the Frobenius-Perron dimension of the object $X$, and the sum
\begin{align*}
\FPdim(\C):=\sum_{X\in\Q(\C)}\FPdim(X)^2
\end{align*}
is defined as the Frobenius-Perron dimension of $\C$.

A fusion category $\C$ is weakly integral if $\FPdim(\C)\in\Z$, $\C$ is integral if $\FPdim(X)\in\Z$ for all $X\in\Q(\C)$. We use $\C_\text{int}$ to denote the maximal integral fusion subcategory of $\C$. The adjoint subcategory $\C_\ad$ of a weakly integral fusion category is always integral \cite{EGNO}.
A simple object of $\C$ is called invertible if $X\otimes X^*=I$, the unit object, equivalently, $\FPdim(X)=1$. A fusion category $\C$ is pointed if $\Q(\C)$ is a finite group, where the group multiplication is induced by the tensor product $\otimes$. In the following, we use $\C_\pt$ to denote the maximal pointed subcategory of $\C$, that is, the fusion subcategory generated by invertible simple objects of $\C$. And we say a  fusion category $\C$ is non-pointed if $\C_\pt\neq \C$. In addition,  two fusion categories $\C$ and $\D$ are  Grothendieck equivalent if  $\Gr(\C)\cong\Gr(\D)$ as fusion rings.

Let $\C$ be a spherical  fusion category $\C$ with spherical structure $j$, which is a natural isomorphism $j=\{j_X|j_X:X\overset{\sim}{\to} X^{**},X\in\C\}$ such that $\dim_j(X)=\dim_j(X^*)$, $\dim_j(X)$ is called  the quantum (or, categorical) dimension of $X$  determined by $j$,  where $\dim_j(X)$ is defined as the (categorical) trace of $\text{id}_X$, that is,
\begin{align*}
\dim_j(X)=\text{Tr}(\text{id}_X):=\text{ev}_X\circ(j_X\circ \text{id}_X)\otimes \text{id}_{X^*}\circ\text{coev}_X,
\end{align*}
where $(X^*,\text{ev}_X,\text{coev}_X)$ is the dual object of $X$ and we suppress the associativity and   unit constraints  of $\C$. We define the global (or, categorical) dimension of the  fusion category $\C$ as
\begin{align*}
\dim(\C):=\sum_{X\in\Q(\C)}\dim_j(X)^2,
\end{align*}
the global dimension $\dim(\C)$ is independent of the choice of the spherical structure of $\C$ and $\dim_j(-)$ induces a homomorphism from $\Gr(\C)$ to $\FC$
\cite[Proposition 4.7.12]{EGNO}.

Given an arbitrary spherical fusion category $\C$,  we can consider the twist (or Galois conjugate) $\C^\sigma$ of $\C$, where $\sigma\in\text{Gal}(\overline{\mathbb{Q}}/\mathbb{Q})$ and $\overline{\FQ}$ is the algebraic closure of $\FQ$. More precisely, $\C^\sigma$ is a fusion category with the same monoidal functor $\otimes$ as $\C$ and the associator of $\C^\sigma$ is obtained by composing  the one of $\C$ with automorphism $\sigma$. Moreover, $\text{dim}(\C^\sigma)=\sigma(\text{dim}(\C))$.
A fusion category $\C$ is said to be pseudo-unitary if $\text{dim}(\C)=\text{FPdim}(\C)$.  For more properties of global dimension, we refer the readers to references \cite{EGNO,ENO1,O3}. In this paper, we will fix a spherical structure $j$ and denote $\dim_j(-)$ by $\dim(-)$.

\subsection{Formal codegrees of fusion categories}
Let $\C$ be a fusion category, then the complexified ring $\Gr(\C)\otimes_\Z\FC$ is semisimple. Given an irreducible representation $\chi$ of $\Gr(\C)\otimes_\Z\FC$,  the element
\begin{align*}
\alpha_\chi=\sum_{X\in\Q(\C)}\text{tr}_\chi(X)X^*
\end{align*}
is central, where $\text{tr}_\chi$ is the ordinary trace function on the representation $\chi$, moreover $\chi'(\alpha_\chi)=0$ if $\chi\ncong\chi'$ and $f_\chi:=\chi(\alpha_\chi)$ is a positive algebraic integer \cite{Lusztig}, $f_\chi$ is called a formal codegree of fusion category $\C$ \cite{O1}. For example, $\FPdim(\C)$ and $\dim(\C)$ are formal codegrees of $\C$.

It was showed that $\frac{\dim(\C)}{f_\chi}$ is always a totally positive algebraic integer
\cite[Corollary 2.14]{O1} and the set of formal codegrees of $\C$ satisfy the following equation \cite[Proposition 2.10]{O1}
\begin{align}\label{classeqt}
\sum_{\chi}\frac{\chi(1)}{f_\chi}=1.
\end{align}
If $\C\ncong\vvec$, then $f_\chi>\sqrt{\frac{2\rank(\C)}{\rank(\C)+1}}\geq\sqrt{\frac{4}{3}}$ for all irreducible representations $\chi$ \cite[Theorem 4.2.1]{O3}. Moreover, $\sigma(\dim(\C))>\frac{4\sqrt{3}}{5}$ if the non-trivial fusion category $\C$ is not a Galois conjugate of $\C(\ssl_2,3)_\ad$ \cite[Proposition A.1.1]{O3}, where $\sigma\in\Gal(\overline{\FQ}/\FQ)$.

Let $\alpha$ be an algebraic integer with the minimal polynomial $g(x)=x^n+a_1x^{n-1}+\cdots+a_{n-1}x+a_n$, then $\alpha$ is called a $d$-number if $(a_n)^j$ divides $(a_j)^n$ for all $1\leq j\leq n$ \cite[Definition 1.1]{O1}, see \cite[Lemma 2.7]{O1} for more equivalent conditions about $d$-numbers. The formal codegrees of a fusion category are $d$-numbers \cite[Theorem 1.2]{O1}, for example.  In addition, the Frobenius-Perron  dimensions and quantum dimensions of  objects, and formal codegrees  of a fusion category are cyclotomic algebraic integers \cite[Corollary 8.53]{ENO1}. Hence, in order to determine whether  a totally positive algebraic integer $\alpha$ is a formal codegree of a fusion category $\C$, we can use the Program GAP to test whether $\alpha$ is a $d$-number and  the Galois group of the minimal polynomial of $\alpha$ is abelian, this is called the $d$-number tests and cyclotomic test \cite{O1}.


\subsection{Modular fusion categories and  representations of $\SL(2,\Z)$}
Let $\C$ be a braided fusion category with braiding $c$ and $\D$  a fusion subcategory of $\C$. Then the centralizer of $\D$ in $\C$ is the following fusion subcategory
\begin{align*}
\D_\C'={\{X\in\C|c_{Y,X} c_{X,Y}=\text{id}_{X\otimes Y},\forall Y\in\D}\}.
\end{align*}
 We call $\C':=\C_\C'$ the M\"{u}ger center of $\C$ \cite{Mu1}.

Let $\C$ be a braided spherical (i.e., pre-modular) fusion category with ribbon structure $\theta$, the matrices $S=(S_{X,Y})$ and $T=(\delta_{X,Y}\theta_X)$ are called the $S$-matrix and $T$-matrix of $\C$, respectively, where $S_{X,Y}=\text{Tr}(c_{Y,X}c_{X,Y})$, $X,Y\in\Q(\C)$.
A modular fusion category is a pre-modular fusion category such that the $S$-matrix $S$ is non-degenerate, equivalently $\C'=\vvec$ \cite{DrGNO2,EGNO,Mu1}. Moreover, the  $S$-matrix a modular fusion category $\C$ determines the multiplication of the Grothendieck ring $\Gr(\C)$ by the famous Verlinde formula \cite{EGNO}, i.e., for $X,Y,Z\in\Q(\C)$,
\begin{align}\label{Verlinde}
N_{X,Y}^Z:=\dim_\FC(\text{Hom}(X\otimes Y,Z))=\sum_{W\in\Q(\C)}\frac{S_{X,W}S_{Y,W}S_{Z^*,W}}{\dim(W)}.
\end{align}

Moreover, given a modular fusion category $\C$,  the Verlinde formula (\ref{Verlinde}) also implies that the set of ring homomorphism from $\Gr(\C)$ to $\FC$ is in bijective correspondence with $\Q(\C)$ \cite{EGNO}. Explicitly, let $X\in\Q(\C)$, then $h_X(Y):=\frac{S_{X,Y}}{\dim(X)}$ defines a ring homomorphism from $\Gr(\C)$ to $\FC$, $\forall Y\in\Q(\C)$; notice that the set of formal codegrees of $\C$ is  $\{\frac{\dim(\C)}{\dim(X)^2}|X\in\Q(\C)\}$ due to the Verlinde formula (\ref{Verlinde}). Since $\sigma\circ h_X(-)$ is also a ring homomorphism  of $\Gr(\C)$, where $\sigma\in\Gal(\overline\FQ/\FQ)$,  there is a unique simple object $Y$ such that $\sigma\circ h_X(-)=h_Y(-)$. Hence, there is a unique  permutation $\hat\sigma$ of $\Q(\C)$ such that $\hat\sigma(X)=Y$ and $\sigma\circ h_X(-)=h_{\hat\sigma(X)}(-)$. We call the subset
\begin{align*}
\Q_X(\C):=\{Y\in\Q(\C)|Y=\hat\sigma(X) ~\text{for some $\sigma\in\Gal(\overline\FQ/\FQ)$}\}
\end{align*}
the Galois orbit of the simple object $X$. When $\Q_I(\C)=\Q(\C)$, then $\C$ is said to be transitive; transitive modular fusion categories are classified   explicitly \cite[Theorem II]{NgWaZh}.

The modular data $(S,T)$ of a modular fusion category $\C$ is also connected closely with the congruence representations of the modular group $\SL(2,\Z)$ \cite{BNRW1,BNRW2,DongLNg}, which  is generated by $\hxs=\left(
                                                      \begin{array}{cc}
                                                        0 & 1 \\
                                                        -1 & 0 \\
                                                      \end{array}
                                                    \right)
$ and $\hxt=\left(
              \begin{array}{cc}
                1 & 1 \\
                0 & 1 \\
              \end{array}
            \right)
$ with relations $\hxs^4=1$ and $(\hxs\hxt)^3=\hxs^2$. Explicitly, the morphism $\hxs\mapsto\frac{1}{\sqrt{\dim(\C)}}S$, $\hxt\mapsto T$ defines a projective representation of $\SL(2,\Z)$ \cite[Theorem 8.16.1]{EGNO}, where $\sqrt{\dim(\C)}$ is the positive square root of $\dim(\C)$.

Let $\tau^\pm(\C):=\sum_{X\in\Q(\C)}\dim(X)^2\theta_X^\pm$
be the Gauss sums of $\C$ \cite{DrGNO2,EGNO}, then $\xi(\C):=\frac{\tau^+(\C)}{\sqrt{\dim(\C)}}$ is called the multiplicative central charge of $\C$. It follows from \cite{BNRW2,DongLNg} that there always exists a third root $\xi$ of  $\xi(\C)$ such that  $s:=\rho_\C(\hxs)=\frac{1}{\sqrt{\dim(\C)}}S$ and $t:=\rho_\C(\hxt)=\frac{1}{\xi}T$ defines a finite-dimensional congruence representation $\rho_\C$ of $\SL(2,\Z)$ of level $N$,  that is, $\ker(\rho_\C)$ is a congruence subgroup of level $N=\text{ord}(t)$. Moreover, $\FQ(S)\subseteq\FQ(T)\subseteq\FQ(t)$
\cite[Theorem II]{DongLNg}. Let $t_X:=\theta_X/\xi$, the normalized ribbon structure of the simple object $X$,  $\forall X\in\Q(\C)$, then we have the Galois symmetry \cite[Theorem II]{DongLNg}, that is, $\sigma^2(t_X)=t_{\hat\sigma(X)}$, $\forall \sigma\in\Gal(\FQ(t)/\FQ)$.

Let $\rho$ be a finite-dimensional congruence representation of $\SL(2,\Z)$ of level $n$, where $n$ is a positive integer. Then $\rho$ factors through the finite groups \begin{align*}
\SL(2,\Z_n)\cong\SL(2,\Z_{p_1^{n_1}})\times\cdots\times\SL(2,\Z_{p_r^{n_r}})
\end{align*} and
$\rho\cong \otimes_{j=1}^r\rho_{p_j}$ by the Chinese Reminder Theorem, where $n=\prod_{j=1}^rp_j^{n_j}$ and $p_j$ are distinct primes, $\rho_{p_j}$ are finite-dimensional representations of subgroups $\SL(2,\Z_{p_j^{n_j}})$. Moreover, given an arbitrary  prime $p$, all finite-dimensional irreducible representations of the group $\SL(2,\Z_{p^m})$ are completely classified and constructed explicitly in \cite{Nob,NobWol}. A  finite-dimensional congruence representation $\rho$ of the modular group $\SL(2,\Z)$  is said to be non-degenerate if the eigenvalues of $\rho(t)$ are distinct; non-degenerate finite-dimensional congruence representations are  irreducible \cite[Lemma 1]{Eholzer2}.  In addition, the set of eigenvalues of $\rho(t)$ is called the $\hxt$-spectrum of $\rho$ following \cite{BNRW2,NgRWW,PSYZ}; we note that the $\hxt$-spectrum of any finite-dimensional irreducible representation of $\SL(2,\Z_{p^m})$ is produced in \cite[Appendix]{PSYZ}.

The following remark is known to experts, we list it here for the reader's convenience, and it will play a key role in the arguments of this paper.
\begin{rema}\label{numGalcon}Let $\C$ be a non-trivial modular fusion category with $N(\dim(\C))=p^N$, where $p$ is an odd prime. Then there exists a simple object $X$ such that $p$ divides $\text{ord}(t_X)$.

Indeed, if $t_Y=1$ for all $Y\in\Q(\C)$, then $\xi=1$ as $t_I=1/\xi=1$, so $\theta_Y=1$ $(\forall Y\in\Q(\C))$. Notice that the balancing equation \cite{EGNO} then implies that $s_{X,Y}=\dim(X)\dim(Y)$, particularly the $S$-matrix of $\C$ can't be non-degenerate, it is a contradiction.

Moreover, let $X$ be the simple object of $\C$ such that $\text{ord}(t_X)=p^n$ is maximal. Then
\begin{align*}
\rank(\C)\geq \frac{p^{n-1}(p-1)}{2}.
\end{align*}
By the Galois symmetry of modular fusion categories \cite{DongLNg}, we have $\sigma^2(t_X)=t_{\hat{\sigma}(X)}$, then \begin{align*}
\Gal(\FQ(t_X)/\FQ)\cdot t_X={\{t_{\hat{\sigma}(X)}=\sigma^2(t_X)|\sigma\in\Gal(\FQ(t_X)/\FQ)}\},
\end{align*}
hence the number of Galois conjugates of $X$ is greater or equal to the order of the following subgroup
\begin{align*}
\Gal(\FQ(t_X)/\FQ)^2:={\{\sigma^2|\sigma\in\Gal(\FQ(t_X)/\FQ)}\}
\end{align*}
of $\Gal(\FQ(t_X)/\FQ)$. It follows immediately that the order of $\Gal(\FQ(t_X)/\FQ)^2$ is exactly $\frac{p^{n-1}(p-1)}{2}$, since $\Gal(\FQ(t_X)/\FQ)$ is a cyclic group of order $p^{n-1}(p-1)$.
\end{rema}
Throughout this paper, we always use $\rho_0$ to denote the trivial representation of $\SL(2,\Z)$.
\begin{exam}\label{example}Let $p\geq5$ be a prime, and let $\C:=\C(\ssl_2,2(p-1))$. Then $\C$ contains a unique non-trivial connected \'{e}tale algebra $A$ \cite[Theorem 6.5]{KiO}, i.e., the regular algebra of Tannakian subcategory $\text{Rep}(\Z_2)\subseteq\C$, such that $\C_A^0$ is a modular fusion category and
\begin{align*}
\dim(\C_A^0)=\frac{\dim(\C)}{\dim(A)^2}=\frac{p}{4\cos^2(\frac{d\pi}{p})}
\end{align*}
by \cite[Theorem 4.5]{KiO}, where $d=\frac{p+1}{2}$. Moreover, the simple objects  of $\C_A^0$ are also characterized explicitly in \cite[Theorem 7.1]{KiO}, i.e., $\Q(\C_A^0)=\{V_0,V_2,\cdots,V_{p-3},V_{p-1}^+,V_{p-1}^-\}$, and
\begin{align*}
\dim(V_j)=\frac{\zeta_{4p}^{j+1}-\zeta_{4p}^{-(j+1)}}{\zeta_{4p}-\zeta_{4p}^{-1}}
=\frac{\cos(\frac{(d-(j+1))\pi}{p})}{\cos(\frac{(d-1)\pi}{p})},0\leq j\leq p-3,
\dim(V_{p-1}^\pm)=\frac{1}{2\cos(\frac{(d-1)\pi}{p})}.
\end{align*}
Hence, the  formal codegrees of $\C(\ssl_2,2(p-1))_A^0$ are
\begin{align*}
p (\text{twice}),\sigma\left(\frac{p}{4\cos^2(\frac{d\pi}{p})}\right)~\text{where}~
\sigma\in\Gal(\FQ(\zeta_p)^+/\FQ).
\end{align*}
In particular, if $p=5$, then $\C(\ssl_2,8)_A^0\cong\C(\ssl_2,3)_\ad\boxtimes\C(\ssl_2,3)_\ad$ as modular fusion category.

By \cite[Theorem 1.17]{KiO}, $\theta_{V_j}=e^{2\pi i\cdot\frac{j(j+2)}{8p}}$ and $\theta_{V_{p-1}^\pm}=\zeta_p^\frac{p^2-1}{8}$. Notice that the multiplicative central charge
$\xi(\C_A^0)=\xi(\C)=e^{2\pi i\cdot\frac{6(p-1)}{16p}}$, let $\xi$ be a third root of $\xi(\C_A^0)$, so $\xi$ is an $24$-th root of $\zeta_p^{3(p-1)}$, while $\zeta_p^{p-1}=\zeta_p^{(p-1)+p(p-1)}=\zeta_p^{p^2-1}$, hence we can choose $\xi=\zeta_p^\frac{p^2-1}{8}$ to be  the third root of $\xi(\C_A^0)$, then
the $\hxt$-spectrum of the normalized $t$-matrix is
\begin{align*}
\left\{\sigma^2(\zeta_p^{-\frac{p^2-1}{8}})\bigg|\sigma\in\Gal(\FQ(\zeta_p)/\FQ)\right\}
\cup\{1(\text{twice})\}.
\end{align*}

Therefore,
the associated modular representation  $\rho_{\C_A^0}\cong\rho_1\oplus\rho_0$, where $\rho_1$ is a $d$-dimensional irreducible representation of $\SL(2,\Z)$ of level $p$ and $\rho_0$ is the trivial representation of $\SL(2,\Z)$.
\end{exam}

\section{Modular fusion categories with  $N(\dim(\C))=p^m$}\label{section3}
In this section, we always assume $p$ is a prime, and we study the structures of modular fusion categories $\C$ such that $N(\dim(\C))$ is a power of $p$. Notice that if $N(\dim(\C))=p$ and $\dim(\C)\notin\Z$, then $\C$ is braided tensor equivalent to a Galois conjugate of $\C(\ssl_2,3)_\ad$ \cite{Schopieray}.

\subsection{Modular fusion category $\C$ with $N(\dim(\C))=p^2$}

Let  $\C$ be a modular fusion category with $N(\dim(\C))=p^2$. Then
\cite[Proposition 4.11]{Schopieray} says that $\frac{1}{2}d_{\dim(\C)}\leq2$, where $d_{\dim(\C)}=[\FQ(\dim(\C)):\FQ]$, that is, $\dim(\C)=1,2,3,4$. Note that $\FQ(\dim(\C))\subseteq\FQ(T_\C)$ and the Cauchy's Theorem \cite[Theorem 3.9]{BNRW1} shows  that $\text{ord}(T_\C)=p^m$ for some positive integer $m$. Since $\FQ(\dim(\C))$ is a real subfield of $\FQ(T_\C)$, $\FQ(\dim(\C))\subseteq\FQ(\zeta_{p^n})^+$, the maximal real subfield of $\FQ(\zeta_{p^n})$, and the Galois group $\text{Gal}(\FQ(\zeta_{p^n})/\FQ)$ is
\begin{align*}
\text{Gal}(\FQ(\zeta_{p^n})/\FQ)\cong(\Z_{p^n})^*=\left\{
                                             \begin{array}{ll}
                                               \Z_{p^{n-1}}\times\Z_{p-1}, & \hbox{$p>2$;} \\
                                               \Z_2, & \hbox{$p=2$ and $n=2$;} \\
                                               \Z_2\times\Z_{2^{n-2}}, & \hbox{$p=2$ and $n\geq3$.}
                                             \end{array}
                                           \right.
\end{align*}
and $[\FQ(\zeta_{p^n}):\FQ(\zeta_{p^n})^+]=2$. Moreover, for any odd prime $p$, it is well-known that the cyclotomic field  $ \FQ(\zeta_{p^n})$ contains a unique quadratic subfield $\FQ(\sqrt{p^*})$, where $p^*:=(-1)^\frac{p-1}{2}p$.

We first  prove the following lemma, which is a direct corollary of
\cite[Theorem 4.4]{Schopieray}.
 \begin{lemm}\label{P^mdegd>2}Let $\C$ be a modular fusion category such that $N(\dim(\C))=p^m$, where $p>3$ is a prime. If $d_{\dim(\C)}>m$, then $m=\frac{p-3}{2}$ and $\C$ is braided equivalent to a Galois conjugate of the transitive modular fusion category $\C(\ssl_2,p-2)_\text{ad}$.
 \end{lemm}
\begin{proof}
Since  $N(\dim(\C))=p^m$ and $ m<d_{\dim(\C)}$, $p$ does not divide $\dim(\C)$ by \cite[Proposition 4.2]{Schopieray}. Thus, for any integer $q\in\Z_{\geq2}$, $q$ does not divides  $\dim(\C)$. Hence, all formal codegrees of $\C$ are Galois conjugates of $\dim(\C)$  \cite[Theorem 4.4]{Schopieray}. Note that $\C$ is a modular fusion category, formal codegrees of $\C$ all have the form $\frac{\dim(\C)}{\dim(X)^2}$ for simple objects $X$ of $\C$, so for any   object $X\in\Q(\C)$, there exists a $\sigma\in\text{Gal}(\FQ(\dim(\C))/\FQ)$ such that
 \begin{align*}
 \frac{\dim(\C)}{\dim(X)^2}=\sigma(\dim(\C))
 =\sigma\left(\frac{\dim(\C)}{\dim(I)^2}\right)=\frac{\dim(\C)}{\dim(\hat{\sigma}(I))^2},
 \end{align*}
 and orthogonality of $S$-matrix of $\C$ means $\hat{\sigma}(I)=X$. Therefore, $\C$ is  a transitive modular fusion category in sense of  \cite{NgWaZh}, then $\C$ must be  a prime transitive modular fusion category for $N(\dim(\C))=p^m$.
Hence it follows from \cite[Theorem 1.1]{NgWaZh} that  $\C\cong\C(\ssl_2,p-2)^\sigma_\text{ad}$ as modular fusion category, where $\sigma\in\Gal(\FQ(\zeta_p)/\FQ)$. Since
 \begin{align*}
 \dim(\C(\ssl_2,p-2)_\text{ad})=\frac{p}{4\text{sin}^2(\pi/p)},~    N\left(\frac{p}{4\text{sin}^2(\pi/p)}\right)=p^{(p-3)/2},
  \end{align*}
  which implies  $m=\frac{p-3}{2}$ and $d_{\dim(\C)}=\frac{p-1}{2}$.
\end{proof}
\begin{coro}\label{P^2degd>2}
Let $\C$ be a modular fusion category such that $N_{\FQ_\C}(\dim(\C))=p^2$. If $d_{\dim(\C)}>2$, then $p=7$ and $\C\cong\C(\ssl_2,5)^\tau_\text{ad}$, where $\tau\in\Gal(\FQ(\zeta_7)/\FQ)$.
 \end{coro}

Let $\C$ be a modular fusion category, $\D:=\boxtimes_{\sigma\in\Gal(\FQ(\dim(\C))/\FQ)}\C^\sigma$, then
\begin{align*}
\dim(\D)=N(\dim(\C))=\Pi_{\sigma\in\Gal(\FQ(\dim(\C))/\FQ)}\dim(\C^\sigma)
=\Pi_{\sigma\in\Gal(\FQ(\dim(\C))/\FQ)}\sigma(\dim(\C)).
\end{align*}
Hence,   \cite[Lemma 4.2.2, Remark 4.2.3]{O3} say that
\begin{align}\label{NormrankIneq}
\rank(\D)=\rank(\C)^{d_{\dim(\C)}}\leq  \dim(\D),
\end{align}
and $\rank(\D)=\dim(\D)$ if and only if $\D$ is pointed. In particular, if $N(\dim(\C))=p^2$, then $\rank(\C)\leq p^2$. When $p=2$, it is easy to see $\dim(\C)=4$, then $\C\cong\I$ is an Ising category or $\C$ is pointed \cite[Example 5.1.2]{O3}; when $p=3$, $d_{\dim(\C)}=1$, that is, $\dim(\C)=9$, then $\C$ is either pointed or $\C$ is braided equivalent to a Galois conjugate of $\C(\mathfrak{so}_5,9,\zeta_{18})_\ad$ by \cite[Theorem]{Yu}, where $\zeta_{18}$ is a primitive $18$-root of unity.

\begin{lemm}\label{degd=2}
Let $\C$ be a modular fusion category such that $N(\dim(\C))=p^2$, where $p$ is a prime. If $d_{\dim(\C)}=2$, then $p=5$, and as a modular fusion category
\begin{align*}
\C\cong\C(\ssl_2,3)_\text{ad}\boxtimes\C(\ssl_2,3)_\text{ad}~\text{or}~ \C\cong\C(\ssl_2,3)_\text{ad}^\sigma\boxtimes\C(\ssl_2,3)_\text{ad}^\sigma,
 \end{align*}
where  $\sigma\in\text{Gal}(\FQ(\zeta_5)/\FQ)$ such that $\sigma(\sqrt{5})=-\sqrt{5}$.
\end{lemm}
\begin{proof}
By using the Cauchy theorem
\cite[Theorem 3.9]{BNRW1} and \cite[Proposition 4.2]{Schopieray}, we know that $\dim(\C)\in\FQ(\zeta_p^s)$ for some positive integer $s$. If $d_{\dim(\C)}=2$, then $\frac{\dim{(\C)}}{p}$ is a totally positive algebraic unit in the unique quadratic subfield $\FQ(\sqrt{p})$ of $\FQ(\zeta_p^s)$, in particular, $p\equiv1 (\text{mod}~4)$ and $N(\frac{\dim{(\C)}}{p})=1$. Let $\sigma\in\text{Gal}(\FQ(\sqrt{p})/\FQ)$ be the unique non-trivial element, then $\sigma(\epsilon_p)\epsilon_p=-1$ since $N_{\FQ(\sqrt{p})}(\epsilon_p)=-1$,  where $\epsilon_p$ is the fundamental unit in the quadratic field $\FQ(\sqrt{p})$.  As $\dim(\C)$ is a totally positive algebraic integer, we have $\sigma(\dim(\C))>0$, so $\frac{\dim{(\C)}}{p}=\epsilon_p^n$ for  a nonzero even integer $n$ by the Fundamental unit theorem \cite[Theorem 11.5.1]{AW}. Besides, for  the Galois conjugate $\sigma(\dim(\C))$ of $\dim(\C)$, we have $\sigma(\dim(\C))>\frac{4\sqrt{3}}{5}$  and $\dim(\C)>\frac{4\sqrt{3}}{5}$ by \cite[Proposition A.1.1]{O3}. Hence without loss of generality, we can assume $n\geq2$ below, then
\begin{align*}
\sigma(\dim(\C))=p\epsilon_p^{-n}>\frac{4\sqrt{3}}{5},
\end{align*}
also $\epsilon_p^{-n}\leq\epsilon_p^{-2}$ as $n\geq2$, thus
\begin{align*}
p\epsilon_p^{-2}\geq p\epsilon_p^{-n}=\sigma(\dim(\C))>\frac{4\sqrt{3}}{5}.
\end{align*}
which then implies $\epsilon_p<\sqrt{\frac{5p}{4\sqrt{3}}}$.

Let $\epsilon_p=\frac{a_p+b_p\sqrt{p}}{2}$ for positive integers $a_p,b_p$. If $b_p\geq2$, then $\epsilon_p>\sqrt{\frac{5p}{4\sqrt{3}}}$; if $b_p=1$, then $a_p^2-pb_p^2=-4$, i.e., $a_p=\sqrt{p-4}$, we also have $\epsilon_p>\sqrt{\frac{5p}{4\sqrt{3}}}$ if $p>5$. In fact, let $p>5$ and $p\equiv1 (\text{mod}~4)$, $\epsilon_p>\sqrt{\frac{5p}{4\sqrt{3}}}$ if and only if $\frac{\sqrt{p-4}}{2}>\left(\sqrt{\frac{5}{4\sqrt{3}}}-\frac{1}{2}\right)\sqrt{p}$, which is equivalent to $\sqrt{\frac{5}{4\sqrt{3}}}-\frac{5}{4\sqrt{3}}>\frac{1}{p}$, while we have inequalities
\begin{align*}
\sqrt{\frac{5}{4\sqrt{3}}}-\frac{5}{4\sqrt{3}}>0.12>\frac{1}{13}\geq\frac{1}{p}, ~\text{when $p\geq13$},
\end{align*}
which is a contradiction for $p\geq13$. Then $p=5$.

If $p=5$ and  $d_{\dim(\C)}=2$, then  $\FQ(\dim(\C))=\FQ(\sqrt{5})$. Since  $\C$ is not a transitive modular fusion category,   $5\mid \FPdim(\C)$ by \cite[Theorem 4.4]{Schopieray}. Meanwhile, \cite[Proposition 4.2]{Schopieray} implies $N(\FPdim(\C))=25$ as $N(\dim(\C))=25$. Therefore, $\FPdim(\C)=5\epsilon_5^n$, where $n$ is a positive even integer. Let $\sigma\in\text{Gal}(\FQ(\zeta_5/\FQ)$ such that $\sigma(\sqrt{5})=-\sqrt{5}$, notice that $\FPdim(\C)$ and $\sigma(\FPdim(\C))$ are formal codegrees of $\C$, so $\sigma(\FPdim(\C))>\sqrt{\frac{4}{3}}$ by \cite[Theorem 4.2.1]{O3}, which implies $n=2$ and $\FPdim(\C)=\frac{15+5\sqrt{5}}{2}$. Since $\rank(\C)\leq 4$ by equation (\ref{NormrankIneq}), by using the argument of \cite[Example 5.1.2(v)]{O3}, we obtain that $\C\cong\C(\ssl_2,3)_\text{ad}\boxtimes\C(\ssl_2,3)_\text{ad}$ or $\C\cong\C(\ssl_2,3)_\text{ad}^\sigma\boxtimes\C(\ssl_2,3)_\text{ad}^\sigma$ as modular fusion category.
\end{proof}
\begin{rema}
 It follows from the proof of Lemma \ref{degd=2} and \cite[Example 5.1.2]{O3} that a non-pointed modular fusion category $\C$ is Grothendieck equivalent to $\C(\ssl_2,8)_A^0$ if and only if $\dim(\C)=5\epsilon_5^m$, where $m=0,\pm2$.
\end{rema}
In summary, Corollary \ref{P^2degd>2} and Lemma \ref{degd=2} imply the following theorem:
\begin{theo}\label{normcases}Let $\C$ be a modular fusion category with $N(\dim(\C))=p^2$, where $p$ is a prime. Then either $\dim(\C)=p^2$, or $p=5,7$. Moreover, if $\dim(\C)\neq p^2$, then
\begin{align*}
\C\cong\left\{
         \begin{array}{ll}
           \C(\ssl_2,3)^\sigma_\text{ad}\boxtimes \C(\ssl_2,3)^\sigma_\text{ad}, & \hbox{if  $p=5$;} \\
          \C(\ssl_2,5)^\tau_\text{ad}, & \hbox{if $p=7$.}
         \end{array}
       \right.
 \end{align*}
as  modular fusion category, where $\sigma\in\Gal(\FQ(\zeta_5)/\FQ)$ and $\tau\in\Gal(\FQ(\zeta_7)/\FQ)$.
\end{theo}

Let $\C$ be a spherical fusion category of global dimension $p$, where $p$ is a prime. Then for any formal codegree $f$ of $\C$, it follows from \cite[Lemma 5.1]{Schopieray} that $f=p\cdot u_f$, where $u_f$ is an algebraic unit in the field $\FQ(\zeta_p)^+$ \cite[Proposition 5.2]{Schopieray}, thus $d_f:=[\FQ(f):\FQ]$ divides $\frac{p-1}{2}$.
\begin{prop}\label{dimpfcodeg}Assume that $\C$ is not pointed. Then $2< d_{\FPdim(\C)}<\frac{p-1}{2}$ if $p>5$.
\end{prop}
\begin{proof}Spherical fusion categories of integer dimensions less or equal to $5$ are classified in \cite[Example 5.2.2]{O3}, so we assume $p>5$ below. If $d_{\FPdim(\C)}=2$, then $u_{\FPdim(\C)}=\epsilon_p^n$ \cite[Theorem 11.5.1]{AW}, where $n\neq0$ is an even integer, as $\C$ is not pointed and  $u_{\FPdim(\C)}$ is a totally positive algebraic unit. Thus $\FPdim(\C)=p\epsilon_p^n$, while $\FPdim(\C)>\sqrt{\frac{8}{5}}$ and $\sigma(\FPdim(\C))>\sqrt{\frac{8}{5}}$ \cite[Theorem 4.2.1]{O3}, where $\langle\sigma\rangle=\Gal(\FQ(\epsilon_p)/\FQ)$. By using the same argument of Lemma \ref{degd=2}, we see $p\leq 11$. However, spherical fusion categories of global dimensions $7,11$ are pointed \cite[Theorem 5.5]{Schopieray}, it is a contradiction. Hence $d_{\FPdim(\C)}>2$ if $p>5$.

If $d_{\FPdim(\C)}=\frac{p-1}{2}$, then \cite[Corollary 2.15]{O2} says
\begin{align*}
\FQ(\FPdim(\C))=\FQ(\zeta_p)^+=\FQ(\dim(X),X\in\Q(\C)).
 \end{align*} That is, each of the homomorphisms $\dim(-)$ and $\FPdim(-)$ has $\frac{p-1}{2}$ Galois conjugates. Since $\C$ is not a pointed fusion category, $\rank(\C)<p$ by
\cite[Remark 4.2.3]{O3}, so $\rank(\C)=p-1$ and the Grothendieck ring $\text{Gr}(\C)$ is commutative. Note that
\begin{align*}
\sum_{\sigma\in\Gal(\FQ(\zeta_p)^+/\FQ)}\frac{1}{\sigma(\dim(\C))}+\frac{1}{\sigma(\FPdim(\C))}=1,
\end{align*}
by \cite[Proposition 2.10]{O1}, then we have
\begin{align*}
\sum_{\sigma\in\Gal(\FQ(\zeta_p)^+/\FQ)}\frac{1}{\sigma(u_{\FPdim(\C)})}=\frac{p+1}{2}.
\end{align*}
While  $\frac{1}{u_{\FPdim(\C)}}\neq1$ is a totally positive  algebraic unit and $\FQ(\zeta_p)^+\neq\FQ(\sqrt{5})$,  the Siegel's trace theorem \cite[Theorem III]{Siegel} shows
 \begin{align*}
\sum_{\sigma\in\Gal(\FQ(\zeta_p)^+/\FQ)}\frac{1}{\sigma(u_{\FPdim(\C)})}
=\text{tr}(\frac{1}{u_{\FPdim(\C)}})>\frac{p-1}{2}\cdot\frac{3}{2}.
\end{align*}
Thus we obtain $\frac{p+1}{2}>\frac{p-1}{2}\cdot\frac{3}{2}$, which then implies  $p<5$, it is impossible.
\end{proof}
\begin{rema}
In fact, let $f_\chi$ be an arbitrary  formal codegree of $\C$, where $\C$ is a spherical fusion category of global dimension $p$ with $p\neq5$. If $f_\chi\notin\Z$, then the method of the above proposition  says that $2<[\FQ(f_\chi):\FQ]<\frac{p-1}{2}$.
\end{rema}
\subsection{Modular fusion category $\C$ with $N(\dim(\C))=p^3$}
Let  $\C$  be a modular fusion category such that $N(\dim(\C))=p^3$ and $\dim(\C)\notin\Z$. Then $d_{\dim(\C)}\leq6$ by
\cite[Proposition 4.11]{Schopieray}, moreover Lemma \ref{P^mdegd>2} says $d_{\dim(\C)}=2,3$. Hence, we know either $4$ or $6$ divides $p-1$. In particular,  there does  not exist modular fusion categories such that $N(\dim(\C))=p^3$ and $d_{\dim(\C)}>1$ when $(p-1,6)=2$. The inequality (\ref{NormrankIneq}) implies $\rank(\C)\leq p-1$, and we know there does not exist such  modular fusion category when $p=2,3$ \cite{O2}, so we assume $p\geq 5$ below. Moreover,
\begin{lemm}
Let $\C$ be a modular fusion category such that $N(\dim(\C))=p^3$ and $d_{\dim(\C)}>1$. Then $\C$ is simple when $p>5$.
\end{lemm}
\begin{proof}Assume  $\C\cong\D\boxtimes\D_\C'$ where $\D$ is a modular fusion subcategory of $\C$, then the proof of  Lemma \ref{P^mdegd>2} shows that either $\D$ or $\D_\C'$ is transitive, as $p$ can not divide both $\dim(\D)$ and $\dim(\D_\C')$. Without loss of generality, assume that $\D$ is braided equivalent to a Galois conjugate of $\C(\ssl_2,p-2)_\text{ad}$,   then $N(\dim(\D))=p^{\frac{p-3}{2}}$ and
\begin{align*}
N_{\FQ(\zeta_p)^+}(\dim(\C))=\left\{
                               \begin{array}{ll}
                                 p^{\frac{3(p-1)}{4}}, & \hbox{if $d_{\dim(\C)}=2$;} \\
                                 p^{\frac{(p-1)}{2}}, & \hbox{if $d_{\dim(\C)}=3$.}
                               \end{array}
                             \right.
 \end{align*}
 thus we obtain $N_{\FQ(\zeta_p)^+}(\dim(\D_\C'))=p^{\frac{(p+3)}{4}}$ if $d_{\dim(\C)}=2$, or $p$ if $d_{\dim(\C)}=3$. Obviously, in each case $p$ can't divide $\dim(\D_\C')$ by \cite[Proposition 4.2]{Schopieray}, so $\D_\C'$ is also transitive by \cite[Theorem 4.4]{Schopieray}. However, there  is a contradiction when $p>5$.
\end{proof}

\begin{lemm}\label{norm=5^3lemm}Let $\C$ be a modular fusion category  such  that  $N(\dim(\C))=5^3$ and $d_{\dim(\C)}\neq 1$.
Then $\dim(\C)=\frac{25\pm5\sqrt{5}}{2}$ or $25\pm10\sqrt{5}$, and $\FPdim(\C)=25+10\sqrt{5}$ or $\frac{25+5\sqrt{5}}{2}$.
\end{lemm}
\begin{proof}It is easy to see  $\FQ(\sqrt{5})=\FQ(\dim(\C))$. Since $d_{\dim(\C)}=2$ and $N\left(\frac{\dim(\C)}{5\sqrt{5}}\right)=-1$, $\dim(\C)=5\sqrt{5}\epsilon^m_5$ with an odd integer $m$, where $\epsilon_5=\frac{\sqrt{5}+1}{2}$ is the fundamental algebraic unit of $\FQ(\sqrt{5})$. By \cite[Proposition A.1.1]{O3}, $\sigma(\dim(\C))>\frac{4\sqrt{3}}{5}$ and $\dim(\C)>\frac{4\sqrt{3}}{5}$, where $\sigma(\sqrt{5})=-\sqrt{5}$. Hence, if $m\leq-1$, then $\dim(\C)=5\sqrt{5}(\frac{\sqrt{5}-1}{2})^{-m}>\frac{4\sqrt{3}}{5}$, i.e., $(\frac{\sqrt{5}-1}{2})^{-m}>\frac{4\sqrt{3}}{25\sqrt{5}}$, thus $m=-1,-3$; if $m>0$, then $\sigma(\dim(\C))=-5\sqrt{5}\sigma(\epsilon^m_5)=5\sqrt{5}\epsilon^{-m}_5>\frac{4\sqrt{3}}{5}$, so $m=1,3$. In summary, $m=\pm1,\pm3$.

Assume that $\dim(-)$ takes value in a  subfield $\mathbb{F}$ of the totally real $\FQ(\zeta_{5^n})^+$   for some positive integer $n$ with $n$ being minimal. If $n>1$, then $[\mathbb{F}:\FQ]\geq10$ as $\FQ(\sqrt{5})\subseteq\mathbb{F}$, for any $\sigma\in\text{Gal}(\mathbb{F}/\FQ)$, there exists a unique simple object $\hat{\sigma}(I)$ such that \begin{align*}\sigma(\dim(\C))=\frac{\dim(\C)}{\dim(\hat{\sigma}(I))^2}, \end{align*}
while $\sigma(\dim(\C))\in{\{\frac{25\pm5\sqrt{5}}{2}}\}$ or ${\{25\pm10\sqrt{5}}\}$, then $\C$ has at least $5$ formal codegrees equal to $\frac{25+5\sqrt{5}}{2}$ and $\frac{25-5\sqrt{5}}{2}$ if $\dim(\C)=\frac{25\pm5\sqrt{5}}{2}$, however $5(1/\frac{25+5\sqrt{5}}{2}+1/\frac{25-5\sqrt{5}}{2})=1$, so these are all formal codegrees of $\C$, it is a contradiction as $n=1$ in this case; if $\dim(\C)=\frac{25\pm5\sqrt{5}}{2}$, then $5(1/(25+10\sqrt{5})+1/(25-10\sqrt{5}))=2$, impossible. Thus, $n=1$.
Since $\frac{\dim(\C)}{\text{FPdim}(\C)}$ is a totally positive algebraic integer that is less than or equal to $1$ \cite[Proposition 9.4.2]{EGNO},  we see $N(\text{FPdim}(\C))=25,125$ by Lemma \ref{P^mdegd>2} and \cite[Proposition 4.2]{Schopieray}.

If $N(\FPdim(\C))=25$, then $\FPdim(\C)=5\epsilon_5^2=\frac{15+5\sqrt{5}}{2}$ and $\C$ is braided equivalent to a Deligne product of two Fibonacci fusion categories by \cite[Example 5.1.2]{O3}, however, in this case $N(\dim(\C))\neq 125$. Therefore, $N(\text{FPdim}(\C))=125$, then $\text{FPdim}(\C)=\frac{25+5\sqrt{5}}{2}$ or $25+10\sqrt{5}$, as $\FPdim(\C)$ is the largest among the set of formal codegrees of $\C$.
\end{proof}
\begin{prop}\label{norm=5^3}Let $\C$ be a modular fusion category such that $N(\dim(\C))=5^3$ and $d_{\dim(\C)}\neq1$. Then as  modular fusion category,
\begin{align*}
\C\cong\C(\ssl_2,3)^{\sigma_1}_\text{ad}\boxtimes\C(\ssl_2,3)^{\sigma_2}_\text{ad}
\boxtimes\C(\ssl_2,3)^{\sigma_3}_\text{ad}  ~\text{or}~ \C\cong\C(\Z_5,\eta)\boxtimes\C(\ssl_2,3)^{\sigma_4}_\text{ad},
\end{align*}
 where $\sigma_i\in\text{Gal}(\FQ(\zeta_5)/\FQ)$ for all $1\leq i\leq4$.
 \end{prop}
 \begin{proof}
By Lemma \ref{norm=5^3lemm}, we know $\FPdim(\C)=25+10\sqrt{5}$ or $\frac{25+5\sqrt{5}}{2}$.

If $\text{FPdim}(\C)=\frac{25+5\sqrt{5}}{2}$, then $\dim(\C)=\frac{25\pm5\sqrt{5}}{2}$, so $\C$ is conjugated to a pseudo-unitary fusion category. Notice that  $\C$ contains a self-dual simple object $X=\hat{\sigma}(I)$ of FP-dimension $\frac{1+\sqrt{5}}{2}$,   $X\otimes X=I\oplus Z$ where  $Z$ is also a self-dual simple object of FP-dimension $\frac{1+\sqrt{5}}{2}$. If $Z=X$, then $\C$ contains $\C(\ssl_2,3)_\ad$ as a fusion subcategory, hence $\C\cong\C(\Z_5,\eta)\boxtimes\C(\ssl_2,3)_\text{ad}$; if not, $X\otimes Z=X\oplus g$, where $g$ is a non-trivial invertible object, so $\C_\text{pt}$ is non-trivial. While $\FPdim(\C_\text{pt})$ divides $\FPdim(\C)$ by \cite[Proposition 8.15]{ENO1}, which  implies  $\FPdim(\C_\text{pt})=5$ and $\C_\text{pt}\cong\C(\Z_5,\eta)$. By \cite[Theorem 3.13]{DrGNO2}, we have
\begin{align*}
&\C\cong\C(\Z_5,\eta)\boxtimes\C(\ssl_2,3)_\text{ad} ~\text{if}\dim(\C)=\frac{25+5\sqrt{5}}{2},\\
&\C\cong\C(\Z_5,\eta)\boxtimes\C(\ssl_2,3)^\sigma_\text{ad} ~\text{if}\dim(\C)=\frac{25-5\sqrt{5}}{2}.
\end{align*}

If $\text{FPdim}(\C)=25+10\sqrt{5}$, then $\C_\text{pt}=\vvec$. In fact, if $\C_\text{pt}\neq\vvec$, then $\C_\pt\cong\C(\Z_5,\eta)$ and  $\C\cong\C(\Z_5,\eta)\boxtimes\D$ with $\dim(\D)=5+2\sqrt{5}$. However, note that $5-2\sqrt{5}$ is a formal codegree of $\D$, which contradicts to \cite[Theorem 1.1.2]{O3}.
If $\dim(\C)=25\pm10\sqrt{5}$, by using the same argument as above, we know $\C$ contains a Galois conjugate of $\C(\ssl_2,3)_\text{ad}$  as a modular fusion subcategory. Consequently, $\C\cong\C(\ssl_2,3)_\text{ad}^\sigma\boxtimes\D$  with $\D$ being a modular fusion subcategory and $N(\dim(\D))=5^2$, where $\sigma\in\Gal(\FQ(\zeta_5)/\FQ)$, and Lemma \ref{degd=2} implies \begin{align*}
&\C\cong\C(\ssl_2,3)_\text{ad}\boxtimes\C(\ssl_2,3)_\text{ad}\boxtimes\C(\ssl_2,3)_\text{ad} ~\text{if}   \dim(\C)=25+10\sqrt{5},\\
&\C\cong\C(\ssl_2,3)^\sigma_\text{ad}\boxtimes\C(\ssl_2,3)^\sigma_\text{ad}
\boxtimes\C(\ssl_2,3)^\sigma_\text{ad}~\text{if}   \dim(\C)=25-10\sqrt{5}.
 \end{align*}

Notice that if $\dim(\C)=\frac{25-5\sqrt{5}}{2}$, then $\dim(\C^\sigma)=\frac{25+5\sqrt{5}}{2}$, hence it suffices to consider when $\dim(\C)=\frac{25+5\sqrt{5}}{2}$ below.
We know that $\C$ contains a self-dual simple object $X=\hat{\sigma}(I)$ with $\dim(X)^2=\frac{3+\sqrt{5}}{2}$, simple objects $Y$ and $Z=\hat{\sigma}(Y)$ such that $\dim(Y)^2=\frac{\dim(\C)}{\text{FPdim}(\C)}=\frac{3-\sqrt{5}}{2}$ and $\dim(Z)^2=\frac{\dim(\C)}{\sigma(\text{FPdim}(\C))}=\frac{7+3\sqrt{5}}{2}$. Let $f$ be another formal codegree of $\C$, then $N(f)=25$ or $125$. If $N(f)=25$, then $f$ is a root of $x^2-ax+25=0$ with $a^2\geq100$, we obtain $f=5$ by  equation (\ref{classeqt}), thus $\C$ contains a simple object $V$ with $\dim(V)^2=\frac{5+\sqrt{5}}{2}$, which is impossible as $\dim(V)\notin\FQ(\zeta_{5^n})^+$. Thus, $N(f)=125$ for all other formal codegrees of $\C$, a direct computation shows $\Q(\C)=\{I,X,Y,Z,V_1,V_2,W_1,W_2\}$ with $\dim(V_1)^2=\dim(V_2)^2=1$ and $\dim(W_1)^2=\dim(W_2)^2=\frac{3+\sqrt{5}}{2}$.

By decomposing $\FPdim(\C)$ into the sum of squares of Frobenius-Perron dimensions of eight simple objects over field $\FQ(\sqrt{5})$,  we find  that $\C$ always contains a simple object of Frobenius-Perron dimension $\frac{1+\sqrt{5}}{2}$, therefore, $\C\cong\C(\ssl_2,3)_\text{ad}\boxtimes\D$  as  modular fusion subcategory by \cite[Theorem 3.13]{DrGNO2}. Thus, again by Lemma \ref{degd=2}
 \begin{align*}
&\C\cong\C(\ssl_2,3)_\text{ad}\boxtimes\C(\ssl_2,3)_\text{ad}\boxtimes
\C(\ssl_2,3)^{\sigma}_\text{ad}~\text{if}\dim(\C)=\frac{25+5\sqrt{5}}{2},\\
&\C\cong\C(\ssl_2,3)^\sigma_\text{ad}\boxtimes
\C(\ssl_2,3)^\sigma_\text{ad}\boxtimes\C(\ssl_2,3)_\text{ad}
~\text{if}\dim(\C)=\frac{25-5\sqrt{5}}{2},
 \end{align*} $\sigma\in\Gal(\FQ(\zeta_5)/\FQ)$ such that $\sigma(\sqrt{5})=-\sqrt{5}$, this completes the proof of the proposition.
 \end{proof}

 Let $\C$ be a modular fusion category. Let $p\geq 5$, and let $\rho_1$ be  a   $\frac{p+1}{2}$-dimensional irreducible representation of $\SL(2,\Z)$ with level equal to $p$, it is proved in \cite[Proposition 3.22]{NgRWW} that if the associated modular representation $\rho_\C$ is equivalent to $m\rho_0\oplus\rho_1$, then $m=1$.

We strengthen the above conclusion in the following proposition and theorem.

\begin{prop}\label{keyprop}
Let $p\geq7$ be a prime, and let $\C$ be a modular fusion category such that $\rho_\C\cong\rho_0\oplus\rho_1$, where  $\rho_1$ is the $\frac{p+1}{2}$-dimensional irreducible representation of $\SL(2,\Z)$ of level $p$. Then  $\C$ is a Galois conjugate of a pseudo-unitary fusion category.
\end{prop}
\begin{proof}Let $\Q(\C)=\{X_0,X_1,\cdots,X_d\}$ and  $S=(S_{X_i,X_j})$ be the un-normalized $S$-matrix of $\C$. Let $d:=\frac{p+1}{2}$, $p^*=\left(\frac{-1}{p}\right)p$ and $\beta:=\left(\frac{a}{p}\right)/\sqrt{p^*}$, where $a$ is an integer comprime to $p$. By \cite[Proposition 3.22]{NgRWW} there exists an real orthogonal matrix
$U=\left(
     \begin{array}{cc}
       A_1 & 0 \\
       0 & A_2 \\
     \end{array}
   \right)
$ such that
\begin{align*}
\rho_\C(\hxs)=\frac{1}{\sqrt{\dim(\C)}}S=U\left(
                                \begin{array}{cc}
                                  1 & 0 \\
                                  0 & \rho_1(\hxs) \\
                                \end{array}
                              \right)
U^T,\rho_\C(\hxt)=\text{diag}(1,\zeta_p^{a\cdot 0},\cdots,\zeta_p^{a\cdot (d-1)^2}),
\end{align*}
  $\rho_1(\hxs)=\beta\left(
                             \begin{array}{cccc}
                               1 & \sqrt{2} & \cdots & \sqrt{2} \\
                               \sqrt{2} &  \\
                               \vdots & &2\cos(\frac{4\pi ajk}{p})   \\
                               \sqrt{2} &  \\
                             \end{array}
                           \right)
$,  $1\leq j,k,l\leq d-1$, $A_1=\left(
                   \begin{array}{cc}
                   a_{11} & a_{12} \\
       a_{21} & a_{22}
                   \end{array}
                 \right)
$, $A_2=\text{diag}(\lambda_1,\cdots,\lambda_{d-1})$, where $\lambda_l\in\{\pm1\}$. Moreover, $\sqrt{2}a_{12},\sqrt{2}a_{22}\in\FQ$, $0<a_{12}^2<1$ and $0<a_{22}^2<1$.

Up to isomorphism, there are exactly two $\frac{p+1}{2}$-dimensional irreducible representations of level $p$ \cite{Eholzer1,NobWol} depending on the value $\left(\frac{a}{p}\right)$. We assume $a=1$ below, the other case is the same. A direct computation shows \begin{align*}
\rho_\C(\hxs)=\left(
\begin{array}{cccccc}
1+a_{12}^2(\beta-1) & a_{12}a_{22}(\beta-1) & \sqrt{2}\beta a_{12}\lambda_1 & \cdots\quad \sqrt{2}\beta a_{12}\lambda_{d-1}\\
a_{12}a_{22}(\beta-1)& 1+a_{22}^2(\beta-1) &\sqrt{2}\beta a_{22}\lambda_1 & \cdots \quad\sqrt{2}\beta a_{22}\lambda_{d-1}  \\
\sqrt{2}\beta a_{12}\lambda_1 &\sqrt{2}\beta a_{22}\lambda_1  &  \\
\vdots &\vdots &  & 2 \beta\lambda_j\lambda_k\cos(\frac{4\pi ajk}{p}) \\
 \sqrt{2}\beta a_{12}\lambda_{d-1}& \sqrt{2}\beta a_{22}\lambda_{d-1}  \\
 \end{array}
 \right),
\end{align*}
$1\leq j,k\leq d-1$. We show that either $X_0$ or $X_1$ don't represent the unit object $I$.

On the contrary, without loss of generality, let $X_0=I$, then $\frac{1}{\sqrt{\dim(\C)}}=|1+(\beta-1)a_{12}^2|$. Notice that
\begin{align*}
\beta=\left(\frac{1}{p}\right)/\sqrt{p^*}=\left\{
                                             \begin{array}{ll}
                                               \frac{1}{p}, & \hbox{$p\equiv1\mod 4$;} \\
                                               \frac{i}{p}, & \hbox{$p\equiv3\mod 4$.}
                                             \end{array}
                                           \right.
\end{align*}
then $|1+(\beta-1)a_{12}^2|^2=\left\{
                                \begin{array}{ll}
1-2a_{12}^2+\frac{p+1}{p}a_{12}^4+2(a_{12}^2-a_{12}^4)\frac{1}{\sqrt{p}}, & \hbox{$p\equiv1\mod 4$;} \\
                                  1-2a_{12}^2+\frac{p+1}{p}a_{12}^4, & \hbox{$p\equiv3\mod 4$.}
                                \end{array}
                              \right.
$, and
\begin{align*}
|\beta-1|^2=\left\{
              \begin{array}{ll}
                \frac{p+1-2\sqrt{p}}{p}, & \hbox{$p\equiv1\mod 4$;} \\
                \frac{p+1}{p}, & \hbox{$p\equiv3\mod 4$.}
              \end{array}
            \right.
\end{align*}
 When $p\equiv3\mod 4$, the set of  formal codegrees of $\C$ is
\begin{align*}
\frac{\dim(\C)}{(\sqrt{\dim(\C)}|S_{1,l}|)^2}\in\{\dim(\C),
\frac{1}{a_{12}^2a_{22}^2|\beta-1|^2}, \frac{p}{(\sqrt{2}a_{12}\lambda_j)^2}, 1\leq j\leq d\}, 1\leq l\leq d+1,
\end{align*}
so  the formal codegrees of $\C$ are all rational, which implies they are  integers. In particular, $\FPdim(\C)=p^m$ by Cauchy's theorem \cite[Theorem 3.9]{BNRW1}, hence $\C$ must be integral as $p$ is odd \cite[Corollary 3.5.8]{EGNO}. However, in this case $\C_\pt$ must be non-trivial, so $\rank(\C)\geq p$, it is impossible.

When $p\equiv1\mod 4$, $\frac{1}{\dim(\C)}=1-2a_{12}^2+\frac{p+1}{p}a_{12}^4+2(a_{12}^2-a_{12}^4)\frac{1}{\sqrt{p}}\notin\Z$, again the Cauchy's theorem \cite[Theorem 3.9]{BNRW1} implies that $N(\dim(\C))=p^m$ for $m\geq2$. Let $\sigma\in\Gal(\FQ(\zeta_p)/\FQ)$ be a generator, then $\sigma$ has no invariant simple objects \cite[Proposition 3.22]{NgRWW}. Hence \begin{align*}
\sigma(\dim(\C))=\frac{\dim(\C)}{\dim(X_2)^2}=\frac{1}{a_{12}^2a_{22}^2(\beta-1)^2}
=\frac{p}{a_{12}^2a_{22}^2}\cdot\frac{(p+1+2\sqrt{p})}{(p-1)^2}.
\end{align*}
Therefore, $N(\dim(\C))=\dim(\C)\sigma(\dim(\C))=\frac{p^2}{(p-1)^2a_{12}^4a_{22}^4}$, that is, $a_{12}^4a_{22}^4(p-1)^2=p^{-m}$ with $2\mid m$ and $m\geq0$, since $a_{12}^2a_{22}^2\in\FQ$.
Assume $m=2l$, so $a_{12}^2a_{22}^2=\frac{1}{p^l(p-1)}$; meanwhile $a_{12}^2+a_{22}^2=1$, so $a_{12}^2$ and $a_{22}^2$ are roots of equation $x^2-x+\frac{1}{p^l(p-1)}=0$, consequently
\begin{align*}
a_{12}^2,a_{22}^2=\frac{\sqrt{(p-1)p^l}\pm\sqrt{(p-1)p^l-4}}{2\sqrt{(p-1)p^l}}.
\end{align*}
Since $p\geq 7$ and $p\equiv1\mod 4$, $((p-1)p^l-4,(p-1)p^l)=4$; as $a_{12}^2,a_{22}^2$ are rational, we have $(p-1)p^l-4=q_1^2$ and $(p-1)p^l=q_2^2$ for nonnegative integers $q_1,q_2$. Note that $(q_2-q_1)(q_2+q_1)=4$, then $l=0$ and $p=5$, it is a contradiction.

Therefore, the unit object $I\in\{X_2\cdots,X_d\}$. Meanwhile the Galois symmetry \cite[Theorem II]{DongLNg} implies that the unit object $I$ has exactly $\frac{p-1}{2}$ Galois conjugates, so $\Q_I(\C)=\{X_2,\cdots,X_d\}$. If the simple object $X_1$ or $X_2$ determines the homomorphism $\FPdim(-)$, then $\FPdim(\C)=\frac{\dim(\C)}{\dim(X_1)^2}$  or $\FPdim(\C)=\frac{\dim(\C)}{\dim(X_2)^2}$. Note that $\FPdim(\C)\in\FQ$ in both cases, again $\C$ must be an integral fusion category, so $\rank(\C)\geq p$, which is absurd. Therefore, the  simple object which determines the homomorphism $\FPdim(-)$ belongs to $\Q_I(\C)$, so $\C$ is  Galois conjugate to  a pseudo-unitary fusion category.
\end{proof}

\begin{theo}\label{rank(p+3)/2}
Let $p\geq5$ be a prime and $\C$  a modular fusion category such that $\rho_\C\cong\rho_0\oplus\rho_1$, where $\rho_1$ is the $d$-dimensional irreducible representation of $\SL(2,\Z)$ of level $p$. Then  $\C$ is Grothendieck equivalent to $\C(\ssl_2,2(p-1))_A^0$. Moreover, $\dim(\C)=p\cdot u$ where $u$ is a Galois conjugate of  algebraic unit $4\cos^2(\frac{d\pi}{p})$.
\end{theo}
\begin{proof} By Proposition \ref{keyprop} we know that $\C$ is a Galois conjugate of a pseudo-unitary fusion category if $p\geq7$. Without loss of generality, we assume that $\C$ is pseudo-unitary.  As we have proved in Proposition \ref{keyprop} that $X_0,X_1\notin\Q_I(\C)$, and $\rho(\hxs)=\frac{1}{\sqrt{\dim(\C)}}S$,
 \begin{align*}
\dim(\C)=\FPdim(\C)=\frac{1}{|S_{X_k,X_k}|^2}~\text{for some $2\leq k \leq d$}.
 \end{align*}
Meanwhile, $\FPdim(\C)$ is the largest among its Galois conjugates, we see $\dim(\C)=\frac{p}{4\cos^2(\frac{d\pi}{p})}$. In addition, if $p=5$, then   the argument of Proposition \ref{keyprop} also shows that $\dim(\C)=5\epsilon_5^{\pm2}$ if the unit object $I\in\{X_0,X_1\}$ and that $\C$ is a Galois conjugate of a pseudo-unitary fusion category if $I\in\{X_2,X_3\}$. Hence, in both cases, we see that  $\C$ is braided tensor equivalent to  a Galois conjugate of $\C(\ssl_2,8)_A^0$ by Lemma \ref{degd=2}.

We assume $p\geq 7$ below, then  $\C$ is a simple modular fusion category. Indeed, if $\C$ is not simple, then fusion subcategories of $\C$ are modular by \cite[Theorem 3.1]{Yu}. Let $\D$ be an arbitrary non-trivial modular fusion subcategory of $\C$, then $\C\cong\D\boxtimes\D_\C'$ by \cite[Theorem 3.13]{DrGNO2}. As $p$ can't divide both $\dim(\D)$ and $\dim(\D_\C')$, therefore $\D$ and $\D_\C'$ are transitive modular fusion categories. By \cite{NgWaZh}, we know $\D\cong\C(\ssl_2,p-2)_\ad^{\sigma_1}$ and $\D_\C'\cong\C(\ssl_2,p-2)_\ad^{\sigma_2}$ for $\sigma_1,\sigma_2\in\Gal(\FQ(\zeta_p)/\FQ)$. However, $\dim((\C(\ssl_2,p-2)_\ad)=\frac{p}{4\sin^2(\frac{\pi}{p})}$ and
\begin{align*}
\dim(\C)
=\dim(\D)\dim(\D_\C')=\sigma_1(\dim(\C(\ssl_2,p-2)_\ad))\sigma_1(\dim(\C(\ssl_2,p-2)_\ad)),
\end{align*}
so $p^\frac{p-1}{2}=N(\dim(\C))=N(\dim(\D))N(\dim(\D_\C'))=p^{p-3}$, i.e., $p=5$, it is impossible.

A direct computation shows that  the set of formal codegrees of $\C$ is
\begin{align*}
\left\{\frac{\dim(\C)}{|\sqrt{\dim(\C)}S_{X_k,X_0}|^2},\frac{\dim(\C)}{|\sqrt{\dim(\C)}S_{X_k,X_1}|^2},
\sigma(\dim(\C)), \forall \sigma\in\Gal(\FQ(\zeta_p)/\FQ)\right\},
\end{align*}
while norms  of $\frac{\dim(\C)}{|\sqrt{\dim(\C)}S_{X_k,X_0}|^2}=\frac{p}{(\sqrt{2}a_{12})^2}$ and $\frac{\dim(\C)}{|\sqrt{\dim(\C)}S_{X_k,X_0}|^2}=\frac{p}{(\sqrt{2}a_{22})^2}$ divide that of $\dim(\C)$, which is power of $p$, consequently $(\sqrt{2}a_{12})^2=\frac{1}{p^{m_1}}$ and $(\sqrt{2}a_{22})^2=\frac{1}{p^{m_2}}$ for nonnegative integers $m_1,m_2$, while $(\sqrt{2}a_{12})^2+(\sqrt{2}a_{22})^2=2$, so $a_{12}=\frac{\mu_1}{\sqrt{2}}$ and $a_{22}=\frac{\mu_2}{\sqrt{2}}$, $\mu_1,\mu_2\in\{\pm1\}$. Hence, we obtain that
\begin{align*}
\rho_\C(\hxs)=\left(
\begin{array}{cccccc}
\frac{\beta+1}{2} & \frac{\mu_1\mu_2(\beta-1)}{2} & \beta \mu_1\lambda_1 & \cdots\quad \beta \mu_1\lambda_{d-1}\\
\frac{\mu_1\mu_2(\beta-1)}{2}& \frac{\beta+1}{2} &\beta \mu_2\lambda_1 & \cdots \quad\beta \mu_2\lambda_{d-1}  \\
\beta\mu_1\lambda_1 &\beta \mu_2\lambda_1  &  \\
\vdots &\vdots &  & 2 \beta\lambda_j\lambda_k\cos(\frac{4\pi ajk}{p}) \\
 \beta \mu_1\lambda_{d-1}& \beta \mu_2\lambda_{d-1} \\
 \end{array}
 \right),1\leq j,k\leq d-1.
\end{align*}
We choose $-a=\frac{p-1}{2}$ below, the other case is same. Since $4\left(\frac{p-1}{2}\right)^3=(p-1)^2\frac{p-1}{2}\equiv\frac{p-1}{2}\mod p$, when $-a=j=k=\frac{p-1}{2}$,
\begin{align*}
|2 \beta\lambda_j\lambda_k\cos(\frac{4\pi ajk}{p})|^2=\frac{4\cos^2(\frac{(d-1)\pi}{p})}{p}
=\frac{1}{\dim(\C)},
\end{align*}
therefore $X_d$ represents the isomorphism class of  the unit object.

Consequently, the set of the Frobenius-Perron dimensions of simple objects of $\C$ is:
\begin{align*}
\left\{\frac{\mu_1\lambda_{d-1}}{2\cos(\frac{(d-1)\pi}{p})},
\frac{\mu_2\lambda_{d-1}}{2\cos(\frac{(d-1)\pi}{p})},
 \frac{\lambda_j\lambda_{d-1}\cos(\frac{j\pi}{p})}{\cos(\frac{(d-1)\pi}{p})},1\leq j\leq d-1\right\},
\end{align*}
since the Frobenius-Perron dimensions of simple objects are positive, $\mu_1=\mu_2=\lambda_j$ for all $1\leq  j\leq d-1$. By comparing the Frobenius-Perron dimensions of simple objects of $\C$ and $\C(\ssl_2,2(p-1))_A^0$ (see Example \ref{example}), let $\phi:\Gr(\C(\ssl_2,2(p-1))_A^0)\to\Gr(\C)$ be a morphism such that $\phi(V_{p-1}^\pm)=X_0,X_1$ and $\phi(V_{2(d-j)})=X_j$, $1\leq j\leq d$, respectively. Notice that $\C$ and $\Gr(\C)$ share the same modular data and  $\phi$ preserves the Frobenius-Perron dimensions of simple objects, hence $\phi$ is an isomorphism of fusion ring by the Verlinde formula (\ref{Verlinde}).
\end{proof}

\begin{coro}\label{ptimesunit}Let $\C$ be a  modular fusion category such that $\dim(\C)=p\cdot u$, where $u\neq1$ and $u\in\FQ(\zeta_p)^+$ is a totally positive algebraic unit. Then $\C$ is Grothendieck equivalent to $\C(\ssl_2,2(p-1))_A^0$.
 \end{coro}
 \begin{proof}Let $\D:=\boxtimes_{\sigma\in\Gal(\FQ(u)/\FQ)}\C^\sigma$. Since $\C$ is not  pointed, the inequality (\ref{NormrankIneq}) and \cite[Remark 4.2.3]{O3} imply
  $\rank(\C)\leq p-1$; so we can assume $p\geq5$.  If $[\FQ(u):\FQ]=2$, then $N(\dim(\C))=p^2$,   $p=5$  and $\dim(\C)=\frac{5(3\pm\sqrt{5})}{2}$ by Lemma \ref{degd=2}. Assume $[\FQ(u):\FQ]>2$ and $p>5$ below. By Remark \ref{numGalcon}, we know $\ord(t_\C)=p$ and at least one of the Galois orbits of simple objects of $\C$ contains $\frac{p-1}{2}$ simple objects. Meanwhile $\C$ can not be a transitive modular fusion category \cite[Theorem I]{NgWaZh}, therefore,  $\rank(\C)>\frac{p-1}{2}$. 

Assume $\rho_\C$ is the modular representation associated to $\C$, then $\rho_\C$ is a direct sum of level $p$ irreducible representations of $\SL(2,\Z)$. However, $\rho_\C$ can't be decomposed as direct sum of irreducible representation of disjoint $\hxt$-spectrum \cite[Lemma 3.18]{BNRW2} and $\rho_\C$ also can't decomposed into direct sum of one-dimensional representations
\cite[Lemma 4]{Eholzer2}. Therefore, by comparing the dimensions of level $p$ irreducible representations of $\SL(2,\Z)$, we obtain that either $\rho_\C$ is an irreducible representation of dimension $p-1$ or $\frac{p+1}{2}$, or $\rho$ is a direct sum of two irreducible representations of dimension $\frac{p-1}{2}$, or $\rho$ is a direct sum of one-dimensional representations and an irreducible representation of dimension $\frac{p+1}{2}$.

If $\rho_\C\cong\rho_1\oplus\rho_2$ with $\dim(\rho_1)=\dim(\rho_2)=\frac{p-1}{2}$, since the $\hxt$-spectrums  of  $\rho_1$ and $\rho_2$ intersect non-trivially \cite[Lemma 3.18]{BNRW2}, $\rho_1=\rho_2$, which is impossible by \cite[Lemma 5.2.2]{PSYZ}. And the irreducible representations of dimension $\frac{p+1}{2}$ can't be realized as representations of modular fusion categories \cite{Eholzer1}. Therefore, $\rho_\C$ is an irreducible representation of dimension $p-1$ or   $\rho_\C=\rho_0\oplus\rho_1$ by \cite[Lemma 3.20, Proposition 3.22]{NgRWW}, where $\rho_0$ is the trivial representation and $\rho_1$ is the irreducible representation of dimension $\frac{p+1}{2}$.

If $\rho_\C$ is irreducible, then $\rank(\C)=p-1$. In particular, the Galois symmetry \cite[Theorem II]{DongLNg} implies that $\Q(\C)$ splits into two orbits and each Galois orbit have exactly $\frac{p-1}{2}$ simple objects. Notice that all formal codegrees of $\C$  are divided by $p$. Indeed, if not, then   $\C$ is a transitive modular fusion category \cite[Theorem 4.4]{Schopieray} and
 \begin{align*}
 \C\cong\C(\ssl_2,p-2)^\sigma_\text{ad}~ \text{ for some $\sigma\in\Gal(\FQ(\zeta_p)/\FQ)$},
  \end{align*}
  however,  $\rank(\C(\ssl_2,p-2))=\frac{p-1}{2}$, it is a contradiction.  Meanwhile, $p^2\nmid f$ for any formal codegree $f$ of $\C$, otherwise, $p^2\mid\dim(\C)$, which implies $p\mid u$, it is impossible by
    \cite[Proposition 4.2]{Schopieray}. Thus, for any formal codegree $f$ of $\C$, we have $f=p\cdot u_f$, where $u_f\in\FQ(\dim(X),X\in\Q(\C))$ is a totally positive algebraic integer as $f\mid \dim(\C)$. Since $f=\frac{\dim(\C)}{\dim(X)^2}$ for some simple object $X\in\Q(\C)$, we see   $u_X:=u_f=\frac{u}{\dim(X)^2}$.

    By definition,
 \begin{align*}
 p\cdot u=\dim(\C)=\sum_{X\in\Q(\C)}\dim(X)^2=\sum_{X\in\Q(\C)}\frac{u}{u_X},
 \end{align*}
 and for any $\sigma\in\Gal(\FQ(\zeta_p)^+/\FQ)$, $\sigma\left(\frac{\dim(\C)}{\dim(X)^2}\right)= \frac{\dim(\C)}{\dim(\hat{\sigma}(X))^2}$, therefore,
 \begin{align*}
 \dim(\hat{\sigma}(X))^2=\frac{u\sigma(\dim(X)^2)}{\sigma(u)}=\frac{u}{\sigma(u_X)}, ~\text{and}~\sigma(u_X)=\frac{u}{\dim(\hat{\sigma}(X))^2}=u_{\hat{\sigma}(X)}.
 \end{align*}
 Thus, we have the following equation
 \begin{align*}
 \sum_{\sigma\in\Gal(\FQ(\zeta_p)^+/\FQ) }\dim(\hat{\sigma}(X))^2= \sum_{\sigma\in\Gal(\FQ(\zeta_p)^+/\FQ) }\frac{u}{\sigma(u_X)}.
 \end{align*}
Hence, let $\Q(\C)=\Gamma_1\cup\Gamma_2$, where $\Gamma_i:={\{\hat{\sigma}(X_i)|\sigma\in\Gal(\FQ(u_{X_i})/\FQ) }\}$, then
 \begin{align*}
 p=\sum^2_{i=1}\sum_{\sigma\in\Gal(\FQ(u_{X_i})/\FQ) }\frac{1}{\sigma(u_{X_i})}.
 \end{align*}
 Since $u_I=u\neq1$,  the Siegel' trace theorem \cite{Siegel}
 says that
 \begin{align*}
 \sum_{\sigma\in\Gal(\FQ(u_{X_i})/\FQ) }\frac{1}{\sigma(u_{X_i})}>1.79\cdot\frac{p-1}{2},
 \end{align*}
except for roots of the equation $x^3-5x^2+6x-1=0$ when $p=7$ \cite[Theorem 1.1]{LM}. If $u_{X_i}=1$, then $\sum_{\sigma\in\Gal(\FQ(u_{X_i})/\FQ) }\frac{1}{\sigma(u_{X_i})}=\frac{p-1}{2}$. Hence,
\begin{align*}
p\geq \frac{5}{3}\cdot\frac{p-1}{2}+\frac{p-1}{2},
  \end{align*}
that is, $p\leq3$, it is impossible.

When $\rho_\C=\rho_0\oplus\rho_1$, we deduce from Theorem \ref{rank(p+3)/2} that $\C$ is Grothendieck equivalent to the modular fusion category $\C(\ssl_2,2(p-1))_A^0$. This finishes the proof of the corollary.
\end{proof}
\begin{coro}
 Let $\C$ be a modular fusion category such that $N(\dim(\C))=p^3$ and $d_{\dim(\C)}=3$. Then $p=7$ and $\C$ is Grothendieck equivalent to  $\C(\ssl_2,12)_A^0$.
\end{coro}

\section{Modular   fusion category of global dimension $p^2$}\label{section4}
In this  section, we always assume $p$ is a prime and we study the structure of modular fusion categories $\C$ of global dimension $p^2$.

\begin{prop}Let $\C$ be a pre-modular fusion category of global dimension $p^2$. If $\C'\ncong\vvec$, then either $\C$ is pointed or $\C\cong\C(\ssl_2,3)_\text{ad}\boxtimes\C(\ssl_2,3)_\text{ad}^\sigma\boxtimes\text{Rep}(\Z_5)$, where $\sigma\in\Gal(\FQ(\zeta_5)/\FQ)$ such that $\sigma(\zeta_5)=\zeta_5^2$.
\end{prop}
\begin{proof}Since $\C'\ncong\vvec$, $\dim(\C')=p,p^2$ by \cite[Theorem 3.1]{Yu}. If $\dim(\C')=p^2$, then $\C$ is symmetric, so it is pointed \cite[Proposition 8.32]{ENO1}. If $\dim(\C')=p$, then $\C'=\text{Rep}(\Z_p)$ is Tannakian,  or $\C'=\text{sVec}$ and $p=2$. When $p=2$, it is obviously $\C$ is pointed \cite[Example 5.1.2]{O3}. If $\C'=\text{Rep}(\Z_p)$, then $\C_{\Z_p}$ is a modular fusion category of dimension $p$, hence $\C_{\Z_p}$ is pointed or $\C_{\Z_p}\cong\C(\ssl_2,3)_\text{ad}\boxtimes\C(\ssl_2,3)_\text{ad}^\sigma$ by
\cite[Theorem 5.12]{Schopieray}. Therefore, $\C$ is a pointed fusion category or $\C\cong\C(\ssl_2,3)_\text{ad}\boxtimes\C(\ssl_2,3)_\text{ad}^\sigma\boxtimes\text{Rep}(\Z_5)$.
\end{proof}


\begin{prop}\label{intdimfsb}Let $\D$ be a non-trivial fusion subcategory of $\C$. If $\dim(\D)\in\Z$, then either $\C$ is pointed, or $\C\cong\C(\ssl_2,2)$, or $\C\cong\C(\ssl_2,3)_\text{ad}\boxtimes\C(\ssl_2,3)_\text{ad}
\boxtimes\C(\ssl_2,3)_\text{ad}^\sigma\boxtimes\C(\ssl_2,3)_\text{ad}^\sigma$, or $\C\cong\C(\Z_5,\eta)\boxtimes \C(\ssl_2,3)_\text{ad}
\boxtimes\C(\ssl_2,3)_\text{ad}^\sigma$, where $\sigma\in\Gal(\FQ(\zeta_5)/\FQ)$ such that $\sigma(\zeta_5)=\zeta_5^2$.
\end{prop}
\begin{proof}
Since $\D$ is a non-trivial fusion subcategory of $\C$ with integral global dimension, by \cite[Theorem 3.1]{Yu}  $\dim(\D)=p$. Then we consider the global dimension of the M\"{u}ger center $\D'$ of $\D$, \cite[Theorem 3.1]{Yu} again implies that either $\D=\D'$ is symmetric or $\D$ is a  modular fusion category. If $\D$ is modular, then $\D\cong \C(\ssl_2,3)_\text{ad}
\boxtimes\C(\ssl_2,3)_\text{ad}^\sigma$ or $\D$ is pointed by \cite[Theorem 5.12]{Schopieray}. Notice that $\C\cong\D\boxtimes\D_\C'$ as modular fusion category \cite[Theorem 3.13]{DrGNO2}, so $\D_\C'$ is also a modular fusion category of dimension $p$, thus the structure of $\C$ is  known. If $\D$ is symmetric, then either it is a Tannakian fusion category or $\D\cong\text{sVec}$. In the first case,  $\D$ must be a Lagrange fusion category as $\D\subseteq\D_\C'$ and $\dim(\D)\dim(\D_\C')=p^2$ \cite[Theorem 3.10]{DrGNO2}, which implies $\C\cong\Y(\text{Vec}^\omega_{\Z_p})$ by \cite[Theorem 4.64]{DrGNO2}, where $\omega\in Z^3(\Z_p,\FC^*)$ is a $3$-cocycle; in the second case  $\dim(\C)=4$, then $\C\cong\C(\ssl_2,2)$ is an Ising category or $\C$ is pointed \cite[Example 5.1.2]{O3}.
\end{proof}

 Let $\C$ be a modular fusion category of global dimension $p^2$, then $\text{ord}(T_\C)$ divides $p^5$ \cite[Corollary 8.18.2]{EGNO}. Since the structures of modular fusion categories of global dimension $4$ and $9$ are known \cite{O3,Yu}, we always assume $p\geq5$ below. Let $\text{ord}(t_\C)=p^n$ where $t$ is the normalized $T$-matrix of $\C$, since $\rank(\C)\leq p^2-1$ \cite[Lemma 4.2.2]{O3},  $n\leq2$ by Remark \ref{numGalcon}. In particular, the  number of simple objects in each Galois orbit is less than or equal to $|\Gal(\FQ(\zeta_{p^2})/\FQ)^2|=\frac{p^2-p}{2}$.

 Assume that $\C$ does not contain a non-trivial fusion subcategory of integer global dimension below, and let $\D$ be a non-trivial modular fusion subcategory of $\C$. Note that we can assume that $\D$ to be a simple modular fusion category and that $\C$ can be  decomposed as Deligne product of simple modular fusion subcategories. Indeed, since $\D'$ is an integral fusion subcategory, $\D'=\vvec$, hence fusion subcategories of $\C$ are modular.
Assume $\dim(\D)$ be the largest among its Galois conjugates, otherwise we can replace $\D$ with one of its Galois conjugates $\D^\sigma$.

 \begin{theo}\label{nintdimfsb}Let $\C$ be a non-simple modular fusion category of dimension $p^2$, where $p>3$. If $\C$ does not  contain a  non-trivial fusion subcategory with integer global dimension, then $\C$ contains a modular fusion subcategory that is braided equivalent to a Galois conjugate of $\C(\ssl_2,p-2)_\text{ad}$.
 \end{theo}
 \begin{proof}
 Let $\D\subseteq\C$ be a simple modular fusion subcategory with $\dim(\D)\notin\Z$, then $\C\cong\D\boxtimes\D_\C'$ and $\D_\C'$ is also a modular fusion category  by \cite[Theorem 3.13]{DrGNO2}. By Lemma \ref{P^mdegd>2}, it suffices to show that $\C$ contains a non-trivial fusion subcategory whose global dimension is  not divided by $p$ in sense of algebraic integers. If $p\nmid\dim(\D)$, then we are done; assume $\dim(\D)=p a_\D$ for some totally positive algebraic integer $a_\D$ below.
If $\D_\C'$ is not simple, let $\D_\C'\cong\A\boxtimes\B$, where $\A,\B$ are non-trivial modular fusion subcategories, obviously $p$ can't divide both $\dim(\A)$ and $\dim(\B)$, thus the argument of Lemma \ref{P^mdegd>2} says that either $\A$ or $\B$ is a  transitive modular fusion category, so $\C$ contains  a Galois conjugate of $\C(\ssl_2,p-2)_\text{ad}$ \cite[Theorem 1.1]{NgWaZh}.

Assume that $\D_\C'$ is a simple modular fusion category such that $\dim(\D_\C')=pb_\D$, $a_\D$  and $b_\D$ are  totally positive  algebraic units with $1=a_\D b_\D$. Notice that $a_\D\in\FQ(\zeta_p)^+$, if not, each Galois orbit of the unit objects $I_\D$ and $I_{\D_\C'}$ has at least $p$ simple objects, then $\rank(\C)\geq p^2$, it is impossible.
By  Corollary \ref{ptimesunit},  $\D$ and $\D_\C'$ are Grothendieck equivalent to modular fusion category $\C(\ssl_2,2(p-1))_A^0$. Up to Galois conjugates, we can assume  $\dim(\D)=\frac{p}{4\cos^2(\frac{d\pi}{p})}$ and $\dim(\D_\C')=4p\cos^2(\frac{d\pi}{p})$ with $d=\frac{p+1}{2}$.

If there exists a $\sigma\in\Gal(\FQ(\zeta_p)/\FQ)$ such that
\begin{align*}
\sigma\left(\frac{p}{4\cos^2(\frac{d\pi}{p})}\right)=\dim(\D_\C')=4p\cos^2(\frac{d\pi}{p}),
 \end{align*}then  $p=5$; in addition, when $p=5$ we have a braided tensor equivalence \begin{align*}
\D\cong\C(\ssl_2,3)_\ad\boxtimes\C(\ssl_2,3)_\ad\quad \text{or}\quad \D\cong\C(\ssl_2,3)_\ad^\sigma\boxtimes\C(\ssl_2,3)_\ad^\sigma
 \end{align*}by Theorem \ref{normcases}, where $\sigma(\zeta_5)=\zeta_5^2$. However, $\C$ must contains a fusion subcategory of dimension $5$ by Proposition \ref{norm=5^3}, which contradicts to the assumption. Therefore, $\C$ contains a modular fusion subcategory that is braided equivalent to a Galois conjugate of $\C(\ssl_2,p-2)_\text{ad}$.
\end{proof}


 \begin{lemm}\label{function}Let $f(x)=\frac{4\pi\sin^2(x)}{x}$, where $x\in(0,\frac\pi5]$. Then $f(x)$ is an increasing function.
\end{lemm}
\begin{proof}We have $f'(x)=\frac{4\pi\sin(x)[2x\cos(x)-\sin(x)]}{x^2}$; since $x\in(0,\frac\pi5]$, $\frac{4\pi\sin(x)}{x^2}>0$. Let \begin{align*}
g(x):=2x\cos(x)-\sin(x),\quad x\in[0,\frac\pi5],
\end{align*}
then $g'(x)=\cos(x)-2x\sin(x)$, $g''(x)=-\sin(x)-2x\cos(x)$, hence $g''(x)<0$ for all $x\in[0,\frac\pi5]$, so $g'(x)\geq g'(\frac\pi5)>0$. Thus $g(x)\geq g(0)=0$ for $x\in[0,\frac\pi5]$, which shows $f'(x)>0$ when $x\in(0,\frac\pi5]$. Consequently, $f(x)$ is a strictly increasing function.
\end{proof}

\begin{theo}\label{transubcat}Let $\C$ be a modular fusion category of global dimension $p^2$, where $p>3$ is a prime. If $\C$ contains a modular fusion category that is a Galois conjugate of $\C(\ssl_2,p-2)_\text{ad}$, then $p\leq23$ and $\FQ(S_\C)\subseteq\FQ(T_\C)=\FQ(\zeta_p)$.
\end{theo}
\begin{proof}Let $\D:=\C(\ssl_2,p-2)_\text{ad}$, then $\C\cong\D\boxtimes\D_\C'$ as modular fusion category \cite[Theorem 3.13]{DrGNO2}, thus $\dim(\D)=\frac{p}{4\text{sin}^2(\pi/p)}$ and $\dim(\D_\C')=4p\text{sin}^2(\pi/p)$. By \cite[Proposition A.1.1]{O3}, we obtain that
$\dim(\D_\C')>\frac{4\sqrt{3}}{5}$,
since $\dim(\D_\C')=f(\pi/p)$, where $f(x)=\frac{4\pi\sin^2(x)}{x}$, and $f(\pi/29)<1.38<\frac{4\sqrt{3}}{5}$,   Lemma \ref{function} implies $p\leq23$.

Since $\C$ is not pointed, $\rank(\C)\leq p^2-1$ by \cite[Remark 4.2.3]{O3}; while $\rank(\D)=\frac{p-1}{2}$, so $\rank(\D_\C')\leq 2(p+1)$. Assume that
$\FQ(t_\C)=\FQ(\zeta_{p^n})$ with $n$ being minimal. When $p>5$, if $n>1$,  then $\rank(\C)>p^2$ by Remark \ref{numGalcon}, which is a contradiction. When $p=5$, the structure of $\C$ is known by Proposition \ref{norm=5^3}, so  $n=1$.
\end{proof}

\begin{lemm}\label{algebraicunit}
Let $\C$ be a modular fusion category of  global dimension $p\beta_p$, where $p$ is a prime and $\beta_p:=2-(\zeta_p+\zeta_p^{-1})$. Then $\dim(X)$ is an algebraic unit for all objects $X\in\Q(\C)$.
\end{lemm}
\begin{proof}
Let $f=\frac{\dim(\C)}{\dim(X)^2}$ be an arbitrary  formal codegree of $\C$ corresponding to a simple object $X$. Since $\C$ is not transitive,  $p\mid f$
by \cite[Theorem 4.4]{Schopieray}, hence $\frac{\beta_p}{\dim(X)^2}$ is an algebraic integer, i.e., $N_{\FQ(\zeta_p)^+}(\dim(X)^2)\mid N_{\FQ(\zeta_p)^+}(\beta_p)$ by \cite[Proposition 4.2]{Schopieray}. While $N_{\FQ(\zeta_p)^+}(\dim(X)^2)=N_{\FQ(\zeta_p)^+}(\dim(X))^2$ and $N_{\FQ(\zeta_p)^+}(\beta_p)=p$, so $N_{\FQ(\zeta_p)^+}(\dim(X)^2)=1$, which means that  $\dim(X)$ must be  an algebraic unit.
\end{proof}
\begin{coro}\label{fixnobject}Let $\D$ be a modular fusion category of global dimension $p\beta_p$ where $p>3$ is a prime. Then $\C$ does not contain simple object that is fixed by $\Gal(\FQ(\zeta_p)^+/\FQ)$.
\end{coro}
\begin{proof}
Let  $X$ be a  simple object that is fixed by the  Galois group $\Gal(\FQ(\zeta_p)^+/\FQ)$, then  for any $\sigma\in\Gal(\FQ(\zeta_p)^+/\FQ)$, we have
\begin{align*}
\frac{\dim(\D)}{\dim(X)^2}=\sigma\left(\frac{\dim(\D)}{\dim(X)^2}\right)
=\frac{\dim(\D)}{\dim(\hat{\sigma}(X))^2}\in\Z,
\end{align*}
which implies $\dim(X)^2=\beta_p$, it contradicts to Lemma \ref{algebraicunit}.
\end{proof}

\begin{prop}\label{restrank}Let $\C$ be a modular fusion category of global dimension $p\beta_p$, where $p\geq7$. Then $\rank(\C)=p-1$  or $\frac{3(p-1)}{2}$.
\end{prop}
\begin{proof}
Since $\FQ(\dim(\C))=\FQ(\zeta_p)^+$ and $\ord(t)=p$ by Theorem \ref{transubcat}, the orbit of $I$ has exactly $\frac{p-1}{2}$ simple objects, it is easy to see that $\C$ can't be a transitive modular fusion category, so $\rank(\C)>\frac{p-1}{2}$.
By Corollary \ref{transubcat}, $p\leq 23$; Corollary \ref{fixnobject} says that each Galois orbit of simple objects of $\C$ has exactly $\frac{p-1}{2}$  simple objects when $p=7,11,23$.

When $p=13,17,19$, for any formal codegree $f$ of $\C$, and let $g(x)=x^n-a_1x^{n-1}+\cdots+(-1)^na_n$ be the minimal polynomial of $f$, where $a_j$ are positive integers and $n$ is a divisor of  $\frac{p-1}{2}$. Notice that $a_0=N(f)$ is a power of $p$ and  that
\begin{align*}
 \sum_{\sigma\in\Gal(\FQ(\zeta_p)^+/\FQ)}\frac{1}{\sigma(\dim(\C))}=\frac{a_{n-1}}{a_n}=\left\{
                                                                           \begin{array}{ll}
                                                                             \frac{7}{13}, & \hbox{if $p=13$;} \\
                                                                             \frac{12}{17},&\hbox{if $p=17$;} \\
\frac{15}{19}, & \hbox{if $p=19$.},
                                                                           \end{array}
                                                                         \right.
\end{align*}
 equation (\ref{classeqt}) implies $\sum_{\sigma\in\Gal(\FQ(f))/\FQ}\frac{1}{\sigma(f)}\leq \frac{a_n-a_{n-1}}{a_n}$.
By using the program GAP,  we can show  that the $d$-number test  or  cyclotomic test fail when $[\FQ(f):\FQ]<\frac{p-1}{2}$ for any formal codegree $f$ of $\C$, hence each Galois orbit of simple object of $\C$ has  $\frac{p-1}{2}$ simple objects. Since $N(\dim(\C))=p^\frac{p+1}{2}$,  the inequality (\ref{NormrankIneq}) shows that
$\rank(\C)\leq \left[p^\frac{p+1}{p-1}\right]=\left[p\cdot p^\frac{2}{p-1}\right]$, where $[\alpha]$ is the integer part of a positive algebraic integer $\alpha$. Therefore, $\rank(\C)\leq 12$ if $p=7$, $\rank(\C)=p-1,\frac{3(p-1)}{2}$ when $p=11,13,17$ and $\rank(\C)=p-1$ when $p=19,23$.

When $p=7$, for any formal codegree $f$ of $\C$, let $g(x)=x^3-ax^2+bx-7^4$ be the minimal polynomial of $f$, where $a,b$ are positive integers. Then the  $d$-number test and cyclotomic test show that
$g(x)=x^3-49x^2+686x-7^4$, which is the minimal polynomial of $7\beta_7$, or $g(x)=x^3-98x^2+1029x-7^4$, the minimal polynomial of the totally positive algebraic integer $\frac{49}{\beta_7^2}$.  Then the set of formal codegrees of $\C$ are exactly the Galois conjugates of $\dim(\C)$ (with multiplicity equals two) and $\FPdim(\C)=\frac{49}{\beta_7^2}$ by  equation (\ref{classeqt}), so $\rank(\C)=9$.
\end{proof}

\begin{lemm}\label{p=11,13}
Let $\C$ be a modular fusion category of  global dimension $p\beta_p$ where $p=11,13$. Then $\rank(\C)\neq \frac{3(p-1)}{2}$.
\end{lemm}
\begin{proof}
Assume $\rank(\C)=\frac{3(p-1)}{2}$. By Proposition \ref{restrank} each Galois orbits of the simple objects has exactly $\frac{p-2}{2}$ simple objects.  Let $X_1=I,X_2,X_3$ be the representatives of each Galois orbits. Let $d:=\frac{p-1}{2}$, let $g_2(x)=x^d+\sum_{j=1}^d(-1)^ja_jx^{d-j}$ and $g_3(x)=x^d+\sum_{j=1}^d(-1)^jb_jx^{d-j}$ be the minimal polynomials of the formal codegrees $\frac{\dim(\C)}{\dim(X_i)^2}$ of $\C$, respectively, where $a_j,b_j$ are positive integers for $1\leq j\leq d$. Since $\dim(X_i)^2$ are algebraic units by Lemma \ref{algebraicunit}, we obtain $a_d=b_d=N(\dim(\C))=p^{d+1}$. The $d$-number condition \cite{O1} says $a_d^{d-1}\mid a_{d-1}^d$ and $b_d^{d-1}\mid b_{d-1}^d$, that is, $p^d$ divides both $a_{d-1}$ and $b_{d-1}$.

Assume $b_{d-1}\geq a_{d-1}$ and  $a_{d-1}=mp^d$ with $m$ being a positive integer. As $\sigma\left(\frac{\dim(\C)}{\dim(X_i)^2}\right)=\frac{\dim(\C)}{\dim(\hat\sigma(X_i))^2}$ for all $\sigma\in\Gal(\FQ(\zeta_p)^+/\FQ)$, then
\begin{align*}
\dim(\C)&=\sum_{i=1}^3\sum_{\sigma\in\Gal(\FQ(\zeta_p)^+/\FQ)}\dim(\hat\sigma(X_i))^2\\
&=\sum_{\sigma\in\Gal(\FQ(\zeta_{17})^+/\FQ)}\frac{\dim(\C)}{\sigma(\dim(\C))}+
\sum_{i=2}^3\sum_{\sigma\in\Gal(\FQ(\zeta_{17})^+/\FQ)}
\frac{\dim(\C)}{\sigma(\dim(\C))}\sigma(\dim(X_i)^2),
\end{align*}
meanwhile $\sum_{\sigma\in\Gal(\FQ(\zeta_p)^+/\FQ)}\frac{\dim(\C)}{\sigma(\dim(\C))}
=\left\{
   \begin{array}{ll}
     5\beta_{11}, & \hbox{if $p=11$;} \\
     7\beta_{13}, & \hbox{if $p=13$.}
   \end{array}
 \right.
$, hence $\frac{6}{p}\geq\frac{2m}{p}$.

We assume $p=11$ below, the other case is same. Let $f_1\leq\cdots\leq f_5$ be the Galois conjugates of $\frac{\dim(\C)}{\dim(X_2)^2}$. If $m = 1$, then $f_j\geq j\cdot 11$ for all $1\leq j\leq 5$, consequently $a_5\geq11^5 \cdot 5!>11^6$, it is impossible. If $m=2$,  the $d$-number test shows $11^2|a_1$, $11^3|a_2$ and $11^4\mid a_4$; also note $f_j\geq\frac{j\cdot 11}{2}$ ($1\leq j\leq 5$), by using a similar restriction of \cite[Theorem 4.2]{Yu}, we see  $27a_1\leq a_2$,  $11a_2\leq a_3$ and $4a_3\leq a_4=2\cdot 11^5$; the cyclotomic test fails for all possible cases, however. If $m=3$, then $f_j\geq\frac{j\cdot 11}{3}$ ($1\leq j\leq 5$),  $18a_1\leq a_2$, $7a_2\leq a_3$ and $7a_3\leq 2a_4=6\cdot 11^5$, again the cyclotomic test fails for all possible cases. Therefore, $\rank(\C)\neq\frac{3(p-1)}{2}$.
\end{proof}

\begin{lemm}\label{p=17}
Let $\C$ be a modular fusion category of  global dimension $17\beta_{17}$, then $\rank(\C)\neq24$.
\end{lemm}
\begin{proof}
On the contrary, assume $\rank(\C)=24$. Let $X_1=I,X_2,X_3$ be the representatives of each Galois orbits. Therefore, same as Lemma \ref{p=11,13}
\begin{align*}
\dim(\C)&=\sum_{\sigma\in\Gal(\FQ(\zeta_{17})^+/\FQ)}\frac{\dim(\C)}{\sigma(\dim(\C))}+
\sum_{i=2}^3\sum_{\sigma\in\Gal(\FQ(\zeta_{17})^+/\FQ)}
\frac{\dim(\C)}{\sigma(\dim(\C))}\sigma(\dim(X_i)^2)\\
&>\sum_{\sigma\in\Gal(\FQ(\zeta_{17})^+/\FQ)}\frac{\dim(\C)}{\sigma(\dim(\C))}
+\frac{\dim(\C)}{M}\sum_{i=2}^3\sum_{\sigma\in\Gal(\FQ(\zeta_{17})^+/\FQ)}\sigma(\dim(X_i)^2),
\end{align*}
where $M$ is the maximal Galois conjugate of $\dim(\C)$. Note that
\begin{align*}
\sum_{\sigma\in\Gal(\FQ(\zeta_{17})^+/\FQ)}\frac{\dim(\C)}{\sigma(\dim(\C))}
=12\beta_{17},
\end{align*}
we have $5\beta_{17}>\frac{\dim(\C)}{M}\cdot2\cdot 8\cdot1.79$ by the Siegel's trace theorem \cite{LM}, which then implies $\sin^2\frac{8\pi}{17}>1 $, it is a contradiction.
\end{proof}

\begin{lemm}\label{nonsimplepbetap}
Let $\C$ be a   modular fusion category of global dimension $p\beta_p$, where $p\geq 7$. If $\C$ is not simple, then $p=7$ and  $\C\cong\C(\ssl_2,5)^\sigma_\text{ad}\boxtimes\C(\ssl_2,5)^{\sigma^2}_\text{ad}$, where $\sigma\in\Gal(\FQ(\zeta_7)/\FQ)$ such that $\sigma(\zeta_7)=\zeta_7^2$.
\end{lemm}
\begin{proof}
Indeed, if $\C\cong\B_1\boxtimes\B_2$ with $\B_1,\B_2$ being non-trivial modular fusion subcategories, then it is easy to see that $p$ cannot divide both  $\dim(\B_1)$ and  $\dim(\B_2)$, so  $\B_1$ and  $\B_2$ are  prime transitive modular fusion categories
\cite[Theorem 4.4]{Schopieray}. Note that
\begin{align*}
N(\dim(\C))=p^\frac{p+1}{2}=N(\dim(\B_1))N(\dim(\B_2))=p^{p-3},
\end{align*}
that is, $p=7$, and $\C\cong\C(\ssl_2,5)^\sigma_\text{ad}\boxtimes\C(\ssl_2,5)^{\sigma^2}_\text{ad}$ as modular fusion category.
\end{proof}
For the modular fusion category $\C(\ssl_2,5)_\ad$,  let $\Q(\C(\ssl_2,5)_\ad)=\{I,X,Y\}$, then the fusion rules $\C(\ssl_2,5)_\ad$ are:
$Y\otimes Y=I\oplus X,X\otimes X=I\oplus X\oplus Y,X\otimes Y=X\oplus Y$. We denote $d_Y:=\FPdim(Y)=2\cos\frac{\pi}{7}$ and $d_X:=\FPdim(X)=d_Y^2-1$ below.
\begin{prop}\label{p=7}
Let $\C$ be a modular fusion category of global dimension $7\beta_7$, then  as a modular fusion category $\C\cong\C(\ssl_2,5)^\sigma_\text{ad}\boxtimes\C(\ssl_2,5)^{\sigma^2}_\text{ad}$.
\end{prop}
\begin{proof}
By Proposition \ref{restrank}, we know $\rank(\C)=9$ and $\FPdim(\C)=\frac{49}{\beta_7^2}$. Let $X_1=I$, and $\Q_{X_1}(\C)=\{X_1,X_2,X_3\}$, that is, $\sigma(\dim(\C))=\frac{\dim(\C)}{\dim(X_2)^2}$, $\sigma^2(\dim(\C))=\frac{\dim(\C)}{\dim(X_3)^2}$, $\sigma\in\Gal(\FQ(\zeta_7)/\FQ)$ such that $\sigma(\zeta_7)=\zeta_7^2$; $\Q_{X_4}(\C)=\{X_4,X_5,X_6\}$, where  $\dim(X_4)^2=1$,
\begin{align*}
\sigma(\dim(\C))=\frac{\dim(\C)}{\dim(X_5)^2},
\sigma^2(\dim(\C))=\frac{\dim(\C)}{\dim(X_6)^2};
\end{align*}
 and $\Q_{X_7}(\C)=\{X_7,X_8,X_9\}$, where $\dim(X_7)^2=\frac{\dim(\C)}{\FPdim(\C)}=\frac{\beta_7^3}{7}$,
 \begin{align*}
\frac{\dim(\C)}{\dim(X_8)^2}=\sigma\left(\frac{\dim(\C)}{\dim(X_7)^2}\right),
\frac{\dim(\C)}{\dim(X_9)^2}=\sigma^2\left(\frac{\dim(\C)}{\dim(X_7)^2}\right).
\end{align*}
Then we obtain that
\begin{align*}
&\epsilon_2\dim(X_2)=\epsilon_5\dim(X_5)=\frac{1}{d_Y},
\epsilon_3\dim(X_3)=\epsilon_6\dim(X_6)=\frac{1}{d_X},\dim(X_4)=\epsilon_4,\\
&\dim(X_7)=\frac{\epsilon_7}{d_Xd_Y},\dim(X_8)=\frac{\epsilon_8 d_Y}{d_X},
\dim(X_9)=\frac{\epsilon_9 d_X}{d_Y},~\epsilon_j\in\{\pm1\},2\leq j\leq 9.
\end{align*}

Since $\hat\sigma(X_j)=X_{j+1}$ and $\hat\sigma^2(X_j)=X_{j+2}$ where $j\in\{1,4,7\}$, for $1\leq k\leq 9$,
\begin{align*}
\sigma\left(\frac{S_{X_j,X_k}}{\dim(X_j)}\right)=\frac{S_{X_{j+1},X_k}}{\dim(X_{j+1})}, \quad \sigma^2\left(\frac{S_{X_j,X_k}}{\dim(X_j)}\right)=\frac{S_{X_{j+2},X_k}}{\dim(X_{j+2})}.
\end{align*}
Meanwhile $\FPdim(\C)=\frac{\dim(\C)}{\dim(X_7)^2}$, so $\FPdim(X_j)=\frac{S_{X_j,X_7}}{\dim(X_7)}$ for all $1\leq j\leq9$.
 Note that $\sigma(d_X)=\frac{1}{d_Y}$ and $\sigma(d_Y)=-\frac{d_X}{d_Y}$, hence $
\FPdim(X_2)=\frac{\dim(X_2)\sigma(\dim(X_7))}{\dim(X_7)}=-\epsilon_2d_Y^2$ and $\FPdim(X_3)=\frac{\dim(X_3)\sigma^2(\dim(X_7))}{\dim(X_7)}=-\epsilon_3d_X^2$,
we see  $\epsilon_2=\epsilon_3=-1$.

Notice that the $S$-matrix of $\C$ are presented by $\epsilon_j$ ($4\leq j\leq 9$), $d_X,d_Y$, $\FPdim(X_4)$, $\FPdim(X_7)$, $S_{X_4,X_4}$ and their Galois conjugates. In particular,
\begin{align*}
\FPdim(X_5)=-\epsilon_4\epsilon_5d_Y^2\sigma(\FPdim(X_4)),
\FPdim(X_6)=-\epsilon_4\epsilon_6d_X^2\sigma^2(\FPdim(X_4)),\\
\FPdim(X_8)=\epsilon_7\epsilon_8d_Y^2\sigma(\FPdim(X_7)),
\FPdim(X_8)=\epsilon_7\epsilon_9d_X^2\sigma^2(\FPdim(X_7)).
\end{align*}
The Verlinde formula (\ref{Verlinde}) implies $X_2\otimes X_2=I\oplus X_3\oplus A$, where $A$ is an object with $\FPdim(A)=2d_X$, so either $A$ is a simple object or $A$ is a direct sum of two simple objects by \cite[Corollary 3.3.16]{EGNO} and \cite[Theorem 1.0.1]{CalMorSny}, note that $A$  contains a simple object of Frobenius-Perron dimension $d_X$ or $d_Y$ as a direct summand in the latter case \cite[Theorem 1.0.1]{CalMorSny}. We claim that $A$ is a direct sum of two non-isomorphic simple objects of Frobenius-Perron dimension $d_X$. Indeed, if $A$ is simple or  $A=V\oplus W$ with $\FPdim(V)=d_Y$ and $\FPdim(W)=2d_X-d_Y$, then  a direct computation  shows  that $\FPdim(\C)>\frac{49}{\beta_7^2}$; if $A=2 M$ with $\FPdim(M)=d_X$, then $M\otimes X_2=2X_2\oplus N$, however $\FPdim(N)=d_Y(d_X-d_Y)<1$, it is a contradiction.

Since the  Frobenius-Perron dimensions of simple objects in the Galois orbits of $X_4$ and $X_7$ are distinct, it is easy to show they are $d_X,d_Y,d_Xd_Y$, respectively. Assume $A=W_1\oplus W_2$ and let $V_1,V_2$  be simple objects of Frobenius-Perron dimension $d_Y$. Then $V_1\otimes V_2$ and $W_1\otimes W_2$ must be simple, as $\C$ does not contain non-trivial invertible simple objects.
Assume $V_1\otimes V_1=I\oplus W_1$,  if $V_2\otimes V_2=I\oplus W_1$, then $2I\subseteq(V_1\otimes V_2)\otimes(V_1\otimes V_2)$, it is impossible, so $V_2\otimes V_2=I\oplus W_2$; note that if $V_1\otimes W_1=V_1\oplus W_2$, then $I\subseteq W_2\otimes (V_1\otimes W_1) \cong V_1\otimes (W_1\otimes W_2)$, it is impossible. Therefore, $V_1\otimes W_1=V_1\oplus W_1$ and $W_1\otimes W_1=I\oplus V_1\oplus W_1$, so  $\C$ contains a  fusion subcategory that is Grothendieck equivalent to $\C(\ssl_2,5)_\ad$. Consequently, $\C\cong\C(\ssl_2,5)^\sigma_\text{ad}\boxtimes\C(\ssl_2,5)^{\sigma^2}_\text{ad}$ as modular fusion category by Lemma \ref{nonsimplepbetap}.
\end{proof}
\begin{prop}\label{rank(p-1)}
Let $\C$ be a modular fusion category of global dimension $p\beta_p$, where $11\leq p\leq 23$. If $\rank(\C)=p-1$, then $p=11$ and $\C$ is braided tensor  equivalent to  a Galois conjugate of modular fusion category $\C( \mathfrak{so}_5,\frac{5}{2})_\ad$.
\end{prop}
\begin{proof}
Let $\rho_\C$ be the associated modular representation of $\C$. We know that $\rho_\C$ can't be decomposed as direct sum of sub-representations with disjoint $\hxt$-spectrums \cite[Lemma 3,18]{BNRW2} and also that $\rho_\C$ can't be decomposed as direct sum of non-degenerate sub-representations of same type  \cite[Lemma 5.2.2]{PSYZ}. Since $\rank(\C)=p-1$ and $\ord(t)=p$, by comparing the dimensions of level $p$ irreducible representations of $\SL(2,\Z)$, we obtain either $\rho_\C$ is irreducible  or $\rho_\C=\rho_1\oplus (\oplus_{m}\rho_0)$, where $\rho_1$ is an irreducible representation of rank $\frac{p+1}{2}$ and $m=\frac{p-3}{2}$. However, \cite[Proposition 3.22]{NgRWW} states $m=1$ and then $p=5$, which is impossible. Hence, $\rho_\C$ is irreducible.

As $\rho_\C$ is non-degenerate and $\rank(\C)\leq 23$, \cite[Main Theorem 4]{Eholzer2} implies that $p=11,17,23$. Moreover, when $p=17,23$, the $(p-1)$-dimensional non-degenerate representations are realized as modular representations of modular fusion categories $\C( \mathfrak{g}_2,\frac{5}{3})$ and $\C(\mathfrak{e}_7,5)_\ad$, respectively. However, neither of the Galois conjugates of $\dim(\C( \mathfrak{g}_2,\frac{5}{3}))$ equal to $17\beta_{17}$,  and  the norm of  $\dim(\C(\mathfrak{e}_7,5)_\ad)$ is $23^{14}$, it is also impossible.
When $p=11$, the $10$-dimensional non-degenerate representations can only be realized as modular representation of  $\C( \mathfrak{so}_5,\frac{5}{2})_\ad$. Therefore, $\C$ is braided tensor equivalent to a Galois conjugate of $\C(\mathfrak{so}_5,\frac{5}{2})_\ad$.
\end{proof}

\begin{theo}\label{centrtrans}Let $\C$ be a modular fusion category of global dimension $p\beta_p$ where $p>3$  is a prime. If $\C_\text{pt}=\vvec$, then $\C$ is braided equivalent to one  of the following modular fusion categories
\begin{align*}
\left\{
         \begin{array}{ll}
           \C(\ssl_2,3)_\text{ad}\boxtimes\C(\ssl_2,3)^\sigma_\text{ad}\boxtimes\C(\ssl_2,3)^\sigma_\text{ad} , & \hbox{ if $p=5;$} \\
           \C(\ssl_2,5)^\tau_\text{ad}\boxtimes\C(\ssl_2,5)^{\tau^2}_\text{ad}, & \hbox{if $p=7;$} \\
           \C(\sso_5,\frac{5}{2})_\ad^\nu, & \hbox{if $p=11$.}
         \end{array}
       \right.
\end{align*}
where $\sigma\in\text{Gal}(\FQ(\zeta_5)/\FQ)$ such that $\sigma(\zeta_5)=\zeta_5^2$, $\tau\in\text{Gal}(\FQ(\zeta_7)/\FQ)$ such that $\tau(\zeta_7)=\zeta_7^2$ and $\nu\in\text{Gal}(\FQ(\zeta_{11})/\FQ)$ such that $\nu(\zeta_{11})=\zeta_{11}^9$.
\end{theo}
\begin{proof}We know $p\leq 23$  by Theorem \ref{transubcat}. When $p=5$, this is the conclusion of Proposition \ref{norm=5^3}, and Proposition \ref{restrank} shows $\rank(\C)=p-1$ or $\frac{3(p-1)}{2}$ if $p>5$. If $\rank(\C)=\frac{3(p-1)}{2}$, then there exists such a modular fusion category $\C$ only for $p=7$ by Lemma \ref{p=11,13}, Lemma \ref{p=17}  and Proposition \ref{p=7}; if $\rank(\C)=p-1$, we have $p=11$ by Proposition \ref{rank(p-1)}.
\end{proof}

Recall that modular fusion categories of global dimension $4$ and $9$ are either pointed, or braided equivalent to a Galois conjugate of Ising category $\C(\ssl_2,2)$ and  $\C(\sso_5,\frac{3}{2})_\ad$ \cite{Yu}, respectively. Hence, combining with the conclusions of \cite{Yu}, Proposition \ref{intdimfsb},   Theorem \ref{nintdimfsb} and Theorem \ref{centrtrans} together imply the following theorem:
\begin{theo}\label{nonsimplep^2}Let  $\C$ be a modular fusion category of global dimension $p^2$, where $p\geq5$ is a prime. If  $\C$ contains a non-trivial fusion subcategory, then either $\C$ is pointed, or $\C$ is braided tensor equivalent to a Galois conjugate of one of the following modular fusion categories
\begin{align*}
\left\{
         \begin{array}{ll}
         \left\{
           \begin{array}{ll}
             \C(\ssl_2,3)_\text{ad}\boxtimes \C(\ssl_2,3)^\sigma_\text{ad}\boxtimes\C(\Z_5,\eta), & \\
           \C(\ssl_2,3)_\text{ad}\boxtimes \C(\ssl_2,3)_\text{ad}\boxtimes\C(\ssl_2,3)^\sigma_\text{ad}\boxtimes \C(\ssl_2,3)^\sigma_\text{ad}, &
           \end{array}
         \right. & \hbox{if  $p=5;$}\\
\C(\ssl_2,5)_\text{ad}\boxtimes\C(\ssl_2,5)^\tau_\text{ad}\boxtimes\C(\ssl_2,5)^{\tau^2}_\text{ad}, & \hbox{if $p=7;$}\\
          \C(\ssl_2,9)_\ad\boxtimes\C(\mathfrak{so}_5,\frac{5}{2})_\ad^\nu & \hbox{if $p=11$.}
         \end{array}
       \right.
\end{align*}
where $\sigma\in\Gal(\FQ(\zeta_5)/\FQ)$ such that $\sigma(\zeta_5)=\zeta_5^2$,  $\tau\in\Gal(\FQ(\zeta_7)/\FQ)$ such that $\tau(\zeta_7)=\zeta_7^2$ and $\nu\in\Gal(\FQ(\zeta_{11})/\FQ)$ such that $\nu(\zeta_{11})=\zeta_{11}^9$.
\end{theo}
This completes the classification of non-simple modular fusion categories of global dimension $p^2$. In addition, it is worth to note that $\C(\sso_5,\frac{3}{2})_\ad$ and  its Galois conjugates are simple modular fusion categories of global dimension $9$.
\begin{ques}\label{quessimpp^2}Let  $p>3$ be a prime. Is there a simple  modular fusion category  $\C$ of global dimension $p^2$?
\end{ques}
Moreover, let $\D$ be a spherical fusion category of global dimension $p$,  then its Drinfeld center $\Y(\D)$ is a modular fusion category of global dimension $p^2$. Therefore, a negative answer to Question \ref{quessimpp^2} will also result in  a complete  classification of spherical fusion categories of prime global dimensions.

\section*{Acknowledgements}The author  is  supported by the National Natural Science Foundation of China (no.12101541), the Natural Science Foundation of Jiangsu Province (no.BK20210785), and the Natural Science Foundation of Jiangsu Higher Institutions of China (no.21KJB110006). The author is grateful to Y. Wang for conversations on the  representations of the modular group $\SL(2,\Z)$.

\bigskip
\author{Zhiqiang Yu\\ \thanks{Email:\,zhiqyumath@yzu.edu.cn}
\\{\small School of Mathematical Sciences,  Yangzhou University, Yangzhou 225002, China}}
\end{document}